%% file: pointwise-relax-SIOPT.tex
\documentclass[onefignum,onetabnum,b5paper]{siamart190516}
 
\input{shared}

\ifpdf
\hypersetup{
	bookmarksdepth=paragraph,
	pdftitle={Convex relaxations of integral variational problems: pointwise dual relaxation and sum-of-squares optimization},
	pdfauthor={Alexander Chernyavsky, Jason J.\ Bramburger, Giovanni Fantuzzi, David Goluskin},
}
\fi

\def\SIOPT{1}

\input{macros}
\newsiamremark{example}{Example}

\begin{document}

\maketitle

\begin{abstract}
\input{abstract}
\end{abstract}



\input{main-text}
\input{appendix}

\bibliographystyle{siamplain}
\bibliography{pdr-no-doi.bib}

\end{document}

%% file: shared.tex
\usepackage{lipsum}
\usepackage{amsfonts}
\usepackage{tabularx}
\usepackage{booktabs}
\usepackage{graphicx}
\usepackage{epstopdf}
\usepackage{algorithmic}
\ifpdf
  \DeclareGraphicsExtensions{.eps,.pdf,.png,.jpg}
\else
  \DeclareGraphicsExtensions{.eps}
\fi

\AtBeginDocument{%
	\setlength{\oddsidemargin}{\dimexpr(\paperwidth-\textwidth)/2-1in}%
	\setlength{\evensidemargin}{\oddsidemargin}%
	\setlength{\topmargin}{%
		\dimexpr(\paperheight-\textheight)/2-\headheight-\headsep-1in}%
}

\usepackage[shortlabels]{enumitem}

\newsiamremark{remark}{Remark}
\newsiamremark{hypothesis}{Hypothesis}
\crefname{hypothesis}{Hypothesis}{Hypotheses}
\newsiamthm{claim}{Claim}

\headers{Convex relaxations of integral variational problems}{A. Chernyavsky, J.\ Bramburger, G. Fantuzzi and D. Goluskin}

\title{Convex relaxations of integral variational problems: pointwise dual relaxation and sum-of-squares optimization\thanks{Submitted to SIAM Journal on Optimization.
\funding{AC, JB, and DG were supported by the Canadian NSERC Discovery Grants Program awards RGPIN-2018-04263, RGPAS-2018-522657 \& DGECR-2018-0037. JB was also supported by a PIMS Postdoctoral Fellowship. GF was supported by an Imperial College Research Fellowship.}}}

\author{Alexander Chernyavsky\footnotemark[5]
\thanks{Department of Mathematics, SUNY Buffalo, USA (\email{chernyav@buffalo.edu}).} 
\and Jason J.\ Bramburger\footnotemark[5] \thanks{Mathematics and Statistics, Concordia University, Canada (\email{jason.bramburger@concordia.ca}).}
\and Giovanni Fantuzzi\thanks{Department of Aeronautics, Imperial College London, UK (\email{giovanni.fantuzzi10@imperial.ac.uk}).}
\and David Goluskin\thanks{Department of Mathematics and Statistics, University of Victoria, Canada (\email{goluskin@uvic.ca}).} 
}

\usepackage{amsopn}


%% file: macros.tex
\definecolor{mygray}{rgb}{.6, .6, .6}
\newcommand\solidrule[1][15pt]{\rule[0.5ex]{#1}{1pt}}
\newcommand\dashedrule{\mbox{%
		\solidrule[3pt]\hspace{3pt}\solidrule[3pt]\hspace{3pt}\solidrule[3pt]}}

\newcommand\dottedrule{\mbox{%
		\solidrule[1.5pt]\hspace{2pt}\solidrule[1.5pt]\hspace{2pt}\solidrule[1.5pt]\hspace{2pt}\solidrule[1.5pt]\hspace{2pt}\solidrule[1.5pt]}}

\newcommand\beq{\begin{equation}}
\newcommand\eeq{\end{equation}}

\newcommand\dx{\mathrm{d}x}

\newcommand\R{\mathbb{R}}

\newcommand\cB{\mathcal{B}}
\newcommand\cD{\mathcal{D}}
\newcommand\cF{\mathcal{F}}
\newcommand\cL{\mathcal{L}}
\newcommand\cQ{\mathcal{Q}}

\newcommand\vphi{\varphi}
\newcommand\diver{\mathrm{div}}
\newcommand{\abs}[1]{\left\vert #1 \right\vert}
\newcommand{\dbound}{\mathrm{d}S}

\newcommand{\Lpdr}{\cL_{\rm pdr}}

\newcommand{\HLpdr}{\widehat{\cL}_{\rm pdr}}
\newcommand{\Linf}{\cL_\infty}
\newcommand{\Lnu}{\cL_\nu}
\newcommand{\HLnu}{\widehat{\cL}_\nu}
\newcommand{\LnuOne}{\cL_{\nu+1}}
\newcommand{\Lsos}{\cL_\nu^{\rm sos}}
\newcommand{\LsosWithArg}[1]{\cL_{#1}^{\rm sos}}
\newcommand{\HLsos}{\widehat{\cL}_\nu^{\rm sos}}
\newcommand{\LsosOne}{\cL_{\nu+1}^{\rm sos}}
\newcommand{\Lsosinf}{\cL_{\infty}^{\rm sos}}

\newcommand{\Fomr}{\cF_{\rm omr}}
\newcommand{\Fmom}{\cF_\nu^{\rm mom}}
\newcommand{\Fmominf}{\cF_\infty^{\rm mom}}
\newcommand{\Fnu}{\cF_\nu}
\newcommand{\Finf}{\cF_\infty}

\newcommand{\Bpdr}{\cB_{\rm pdr}}
\newcommand{\HBpdr}{\widehat{\cB}_{\rm pdr}}
\newcommand{\Bsos}{\cB_\nu^{\rm sos}}
\newcommand{\HBsos}{\widehat{\cB}_\nu^{\rm sos}}

\def\cwgeq{\rotatebox[origin=c]{-90}{$\geq$}}
\def\cweq{\rotatebox[origin=c]{-90}{$=$}}
\usepackage{enumitem}
\usepackage{tikz}
\usetikzlibrary{cd}
\usepackage{amssymb}

%% file: abstract.tex
We present a method for finding lower bounds on the global infima of integral variational problems, wherein $\int_\Omega f(x,u(x),\nabla u(x)){\rm d}x$ is minimized over functions $u\colon\Omega\subset\mathbb{R}^n\to\mathbb{R}^m$ satisfying given equality or inequality constraints. Each constraint may be imposed over $\Omega$ or its boundary, either pointwise or in an integral sense. These global minimizations are generally non-convex and intractable. We formulate a particular convex maximization, here called the pointwise dual relaxation (PDR), whose supremum is a lower bound on the infimum of the original problem. The PDR can be derived by dualizing and relaxing the original problem; its constraints are pointwise equalities or inequalities over finite-dimensional sets, rather than over infinite-dimensional function spaces. When the original minimization can be specified by polynomial functions of $(x,u,\nabla u)$, the PDR can be further relaxed by replacing pointwise inequalities with polynomial sum-of-squares (SOS) conditions. The resulting SOS program is computationally tractable when the dimensions $m,n$ and number of constraints are not too large. The framework presented here generalizes an approach of Valmorbida, Ahmadi, and Papachristodoulou [\href{https://doi.org/10.1109/TAC.2015.2479135}{IEEE Trans.\ Automat.\ Contr., 61:1649--1654 (2016)}].
We prove that the optimal lower bound given by the PDR is sharp for several classes of problems, whose special cases include leading eigenvalues of Sturm--Liouville problems and optimal constants of Poincar\'e inequalities. For these same classes, we prove that SOS relaxations of the PDR converge to the sharp lower bound as polynomial degrees are increased. Convergence of SOS computations in practice is illustrated for several examples.

%% file: main-text.tex

\section{\label{sec:intro}Introduction}

Finding the globally optimal solution to non-convex integral variational problems is often intractable, even by computational methods. One way to study such problems is by convex relaxation---that is, by formulating an easier convex problem whose optimum is a one-sided (and perhaps sharp) bound on the global optimum of the original problem. In the present work, we study and apply a general approach to relaxing integral variational problems. The relaxation can be chosen to give a convex problem of either infinite or finite dimension, and in many cases the finite-dimensional relaxations can be solved computationally using tools of polynomial optimization. A broad class of variational problems to which this approach is applicable includes minimizations of the form
\begin{equation}
\mathcal{F}^*
:= \inf_{\substack{u\in W^{1,p} \\ +\text{ constraints}}}
	\int_\Omega f(x,u,\nabla u)\,\dx,
\label{eq:primal-vp}
\end{equation}
where $\Omega\subset\R^n$ is open and bounded, $W^{1,p}$ is the Sobolev space of functions $u:\Omega\to\R^m$ such that all components of $u$ and their weak derivatives are integrable in the Lebesgue space $L^p$ for some $p\ge1$, and $\nabla u:\Omega\to\R^{m\times n}$ is the gradient tensor defined in the weak sense. The integrand function $f$ can be any such that the integral is well defined for all admissible $u$, and, as specified in the next section, the constraints can include pointwise or integral conditions over $\Omega$ or its boundary $\partial\Omega$. In general the infimum $\cF^*$ might not be attained, meaning no global minimizer $u$ exists.

An \emph{upper} bound on \cref{eq:primal-vp} can be found by evaluating the integral for any particular $u$ satisfying the constraints, and local minimizers for the constrained functional can be computed using gradient descent methods or by solving optimality conditions such as Euler--Lagrange equations. When~\cref{eq:primal-vp} is a non-convex problem, however, any $u$ that is a local minimizer might not be a global minimizer. Even when global minimizers are known to exist, in general finding them is beyond current techniques. Here we pursue \textit{lower} bounds on the global infimum $\cF^*$, which requires a different approach.

\subsection{\label{sec: sketch}Sketch of the method}

The strategy we follow is dual to the measure-theoretic approach of \cite{Korda2018}, and it is reminiscent of the calibration method \cite[\S1.2]{Buttazzo1998} and the translation method \cite{Firoozye1991} in the calculus of variations. It consists of two steps that, as explained shortly, can be seen as a \emph{relaxation}\footnote{By a `relaxation' of an optimization problem, we mean an easier problem that may have the same optimum or may give a one-sided bound on the original optimum. This differs from `relaxations' in variational analysis, which do not change the optimum but ensure existence of optimizers.} of the Lagrangian dual problem to the original minimization \cref{eq:primal-vp}. The first step is to change the functional without changing its value by adding terms that, by the divergence theorem, integrate to zero. The term added to $f(x,u,\nabla u)$ in the volume integral is a total divergence of the form $\diver_x\vphi(x,u(x))$ for some $\vphi:\Omega\times\R^m\to\R^n$. In general an equal boundary integral must be subtracted, as in the next section's formulation, but often this can be avoided by choosing $\vphi$ so that the boundary integral vanishes for all admissible $u$. The second step is to estimate the new integral from below by taking the pointwise infimum of its integrand over all $(u,\nabla u)$ arguments. In the simpler case with no boundary integral and no side constraints on $u$ and $\nabla u$, these two steps give
\beq
\label{eq: relaxation simple}
\int_\Omega f(x,u,\nabla u)\,\dx=\int_\Omega [f+\cD\vphi](x,u,\nabla u)\,\dx\ge \int_\Omega \inf_{\substack{y\in\R^m\\z\in\R^{m\times n}}} [f+\cD\vphi](x,y,z)\,\dx,
\eeq
where $\cD\vphi:\Omega\times\R^m\times\R^{m\times n}\to\R$ satisfies $\cD\vphi(x,u,\nabla u)=\diver_x\vphi(x,u(x))$ by virtue of its definition,
\vspace*{-10pt}
\begin{equation}
\label{eq: D}
\cD\varphi(x,y,z) := \sum_{i=1}^n \frac{\partial\varphi_i}{\partial x_i}
+ \sum_{i=1}^n \sum_{j=1}^m z_{ji}\frac{\partial\varphi_i}{\partial y_j}.
\end{equation}
The $z_{ji}$ terms appearing linearly in $\cD\vphi(x,y,z)$ correspond to $\nabla u$ components that arise when $\diver_x\vphi(x,u(x))$ is expanded by the chain rule. This is the only remnant of the relationship between $u$ and $\nabla u$ on the right-hand side of~\cref{eq: relaxation simple}, since any values $(y,z)\in\R^m\times\R^{m\times n}$ are allowed at each point $x\in\Omega$. The more general formulation in the next section shows how constraints from~\cref{eq:primal-vp} can be incorporated into~\cref{eq: relaxation simple}.

The lower bound in \cref{eq: relaxation simple} might seem overly conservative, but in various cases there exist $\vphi$ for which it is an equality, as we prove here for three classes of variational problems. For a simple example with scalar $u$ on the one-dimensional domain $\Omega=(-\pi/3,\pi/3)$, consider the minimization~\cref{eq:primal-vp} over $u$ that vanish at the boundary points and with integrand $f=u_x^2-u^2-2u$. For the bad choice of $\vphi=0$, the right-hand side of \cref{eq: relaxation simple} is $-\infty$ because $[f+\cD\vphi](x,y,z)=z^2-y^2-2y$ is not bounded below over $(y,z)\in\R\times\R$. However, \cref{ex:simple-variational-problem} below shows that, for the optimal choice of $\vphi(x,u)=2u^2\sin x/(2\cos x - 1)$, the right-hand side of~\cref{eq: relaxation simple} is equal to $2(\tfrac{\pi}{3}-\sqrt{3})$, which is exactly the minimum of the variational problem.

In general, there are infinitely many possible choices of $\vphi$ giving different lower bounds via \cref{eq: relaxation simple}, so it is natural to maximize the right-hand side of~\cref{eq: relaxation simple} over $\vphi$.
By minimizing each expression in \cref{eq: relaxation simple} over $u$ and then maximizing over all functions $\varphi$ in some yet-unspecified set $\Phi$, we obtain a sketch of our approach in the simpler case without constraints or boundary integrals:
\beq
\cF^* = \inf_{u\in W^{1,p}}\int_\Omega f(x,u,\nabla x)\,\dx \ge \sup_{\vphi\in\Phi}\int_\Omega \inf_{\substack{y\in\R^m\\z\in\R^{m\times n}}} [f+\cD\vphi](x,y,z)\,\dx.
\label{eq: relaxation simple 2}
\eeq
Standard Lagrangian duality (see, for instance,~\cite{Ekeland1999}) gives a possibly different lower bound $\cF^* \geq \cL^*$ that is also in the form of a maximization problem with supremum $\cL^*$. Like the original minimization, this Lagrangian dual problem is generally intractable, but it may be {relaxed} to obtain easier maximization problems whose maxima are lower bounds on $\cL^*$. The right-hand maximization in~\cref{eq: relaxation simple 2} is a particular relaxation of the Lagrangian dual; we call it the \emph{pointwise dual relaxation} (PDR) and denote it by $\Lpdr$. The PDR and its relation to the Lagrangian dual are described in \cref{sec:pointwise-dual} in a level of generality that allows for constraints and boundary integrals.

The right-hand PDR in~\cref{eq: relaxation simple 2} is still too hard to solve exactly in general, although we give solutions here for some cases. However, further relaxation leads to computationally tractable maximization problems for the broad class of problems with `polynomial data'---meaning that the integrand $f$ and all constraints are polynomial in the components of $(x,u,\nabla u)$. Seeking $\vphi(x,u)$ from a finite space of polynomials of degree $\nu$ or less relaxes the PDR into a linear optimization problem whose constraints amount to pointwise nonnegativity of polynomial expressions. The optimal value $\Lnu$ of this linear problem is a lower bound on $\Lpdr$ and, consequently, on $\cF^*$. Although pointwise nonnegativity of a multivariate polynomial is NP-hard to decide in general~\cite{Murty1987}, it can be enforced by the stronger condition that the polynomial is representable as a sum of squares (SOS) of other polynomials. Imposing these SOS conditions leads to an \emph{SOS program}---a convex maximization problem with SOS constraints and with tunable parameters appearing only linearly in the constraints and in the optimization objective. As detailed in \cref{sec:SOS}, the optimal value $\Lsos$ of this SOS program is a lower bound on $\Lnu$. Since SOS programs can be reformulated as semidefinite programs (SDPs) and solved numerically~\cite{Nesterov2000, Parrilo2000, Lasserre2001}, it is often tractable to compute the value $\Lsos$ and therefore obtain a numerical lower bound on $\cF^*$.

\subsection{Related work}\label{sec:related-work}

The approach sketched above to find lower bounds on the global minimum of a variational problem extends ideas used in~\cite{Valmorbida2015a, Valmorbida2016,Ahmadi2014b,Ahmadi2015,Ahmadi2016,Ahmadi2017,Ahmadi2019} to verify integral inequalities. Separately, \cite{Korda2018} proposed a measure-theoretic way to relax variational problems, wherein the original infimum $\cF^*$ is bounded below by the infimum $\Fomr$ of a convex minimization problem formulated using so-called occupation measures and boundary measures. For variational problems with polynomial data, $\Fomr$ can be bounded below by an infimum $\Fnu$ over sequences of moments truncated at degree $\nu$, which is in turn bounded below by an infimum $\Fmom$ over so-called pseudomoments that can be computed numerically via semidefinite programming.

\enlargethispage{\baselineskip}
The infima of the measure-theoretic relaxations of~\cite{Korda2018} are related to the suprema of the dual relaxations sketched in \cref{sec: sketch} by the inequalities
\begin{alignat}{13}
\label{eq:duality-relation-of-bounds}
\cF^* \; && \geq \; && \Fomr \; && \ge \; && \Fnu \; && \ge \; && \Fmom  \nonumber \\
\cwgeq \;\; && && \cwgeq \;\;\; && && \cwgeq \;\, && &&\cwgeq \quad\;
\\ \nonumber
\cL^*\; && \ge \; && \Lpdr \; && \ge \; && \Lnu \; && \ge \; && \Lsos.\;\;
\end{alignat}
Horizontal inequalities reflect relaxations: moving from left to right, constraints become less restrictive in the minimizations of the top row and more restrictive in the maximizations of the bottom row. Each vertical inequality reflects weak Lagrangian duality, meaning that the two problems are related by swapping the order of an infimum and supremum, hence the inequality. The duality $\cF^*\ge\cL^*$ is what defines $\cL^*$, and $\Fmom \geq \Lsos$ reflects the well-known duality between SOS programs and pseudomoment problems \cite{Lasserre2001,Laurent2009,Lasserre2015}.
The inequality $\Fomr \geq \Lpdr$ is shown in~\cite{Fantuzzi2022}, and the same arguments carry over to show that $\Fnu \geq \Lnu$.

The only quantities in \cref{eq:duality-relation-of-bounds} that are typically tractable to compute are $\Fmom$ and $\Lsos$, both of which are defined only in the case of polynomial data. In \cref{sec: formulation} we formulate the successive relaxations in the bottom row of \cref{eq:duality-relation-of-bounds} at a level of generality that allows for boundary integral terms as well as pointwise and integral constraints. The works~\cite{Korda2018,Fantuzzi2022} formulate and study $\Fomr$, from which $\Fnu$ and $\Fmom$ follow as described in~\cite{Korda2018}. To bound $\cF^*$ in applications, there is little difference in beginning with $\Lsos$ or with $\Fmom$ because both quantities are calculated as solutions to a pair of dual SDPs that are solved simultaneously by standard primal-dual algorithms.

For variational problems \cref{eq:primal-vp} in general, it is largely an open challenge to characterize when the various inequalities in \cref{eq:duality-relation-of-bounds} are or are not strict, especially as one takes $\nu\to\infty$ for the polynomial/moment degree in the four right-most quantities. (With finite $\nu$, the strict inequalities $\Lpdr>\Lnu$ and $\Fomr>\Fnu$ are typical.) Existing partial results concern whether the relaxation from $\cF^*$ to $\Fomr$ introduces a relaxation gap (i.e., a strict inequality) and whether the six right-most quantities in \cref{eq:duality-relation-of-bounds} are equal in the infinite-$\nu$ limit. In~\cite{Fantuzzi2022,Korda2022} it is proved that $\cF^*=\Fomr$ if the integrand $f$ satisfies certain convexity assumptions, but examples where $\cF^*>\Fomr$ are given also. As for the six right-hand quantities in \cref{eq:duality-relation-of-bounds}, in the infinite-$\nu$ limits (denoted by a subscript $\infty$) the equalities
\begin{alignat}{13}
\label{eq:duality-relation-compact}
\Fomr \; && = \; && \Finf \; && = \; && \Fmominf  \nonumber \\
\cweq \;\, && && \cweq \;\;\;\; && &&\cweq \quad\;\;
\\ \nonumber
\Lpdr \; && = \; && \Linf \; && = \; && \Lsosinf \;\;
\end{alignat}
have been proved under certain conditions. When constraints place $(x,u,\nabla u)$ values in a compact set, the left-hand equalities in the top and bottom lines of \cref{eq:duality-relation-compact} are guaranteed by the Weierstrass approximation theorem.
The \emph{strong duality} statement $\Fomr=\Lpdr$ is proved in \cite{Fantuzzi2022} under certain coercivity conditions, and their arguments immediately extend to establish that $\Fnu=\Lnu$ for every $\nu$, which in the infinite-$\nu$ limit gives the second vertical equality in \cref{eq:duality-relation-compact}.
In the case of polynomial data, for which the two right-hand quantities in \cref{eq:duality-relation-compact} are defined, the right-hand equalities in the top and bottom lines are proved in \cite[Theorem~3]{Korda2018} and our \cref{th:sos-variational-problem}, respectively, under a condition slightly stronger than the compactness of $(x,u,\nabla u)$ values (cf.\ \cref{rem:compactness}). This condition is stronger than those under which the other equalities are proven, so it suffices for all equalities in \cref{eq:duality-relation-compact}. Even outside of these conditions, we are not aware of counterexamples to any equalities in \cref{eq:duality-relation-compact}.

\subsection{Contributions}

The present work makes three main contributions. First, by making precise the ideas sketched above, \cref{sec: formulation} explains how to formulate pointwise dual relaxations that bound global optima of integral variational problems. In the case of polynomial data, we explain how further relaxation using SOS conditions gives computationally tractable SOS programs. Our relaxations are only slightly more general than the one- and two-dimensional formulations in~\cite{Valmorbida2016,Valmorbida2015a,Ahmadi2014b,Ahmadi2015,Ahmadi2016,Ahmadi2017,Ahmadi2019}, but this generality makes them dual to the measure-theoretic approach of~\cite{Korda2018}. 
We also generalize this framework to
families of integrands depending linearly on tunable parameters $\lambda\in\R^\tau$. This gives a way to relax optimizations over $\lambda$ where parametrized variational problems appear in \emph{constraints} such as $\cF^*(\lambda)\ge0$.
For the case of polynomial data, in \cref{sec:SOS} we prove general results regarding the convergence of the lower bounds $\Lsos$ to the PDR value $\Lpdr$ as $\nu$ is raised, along with analogous results for the generalized framework where variational problems appear in constraints.

Our second contribution is to prove in \cref{sec: sharp} that the PDR is sharp, meaning that $\cF^*=\Lpdr$ and likewise in the generalized framework, for three classes of problems that include Sturm--Liouville problems and Poincar\'e inequalities. We also show that further SOS relaxations are sharp for problems in these classes with polynomial data.

Finally, in \cref{sec:computational} we illustrate the numerical convergence of SOS programs that relax PDRs. We choose three examples with features not present in most previous computations~\cite{Valmorbida2015a,Valmorbida2016,Ahmadi2014b,Ahmadi2015,Ahmadi2016,Ahmadi2017,Ahmadi2019,Korda2018}, namely more complicated geometries or non-convex integrands. The observed convergence is guaranteed by our theoretical results for the first two examples---optimal constants of the Poincar\'e inequality for the $L^2$ norm on 2D domains with corners and for the $L^4$ norm on a 1D interval---but not for a non-convex example from~\cite{Pedregal2000} whose minimum is not attained.

\section{\label{sec: formulation}Pointwise dual relaxations}

The PDR approach for finding lower bounds on the integral variational problem~\cref{eq:primal-vp} can incorporate equality or inequality constraints on $u$ over the domain and/or its boundary, including nonlinear integral constraints, pointwise constraints, and boundary conditions. \Cref{ss:variational-problems} makes precise the 
ideas sketched in the introduction, arriving at a PDR formulation that includes constraints. \Cref{ss:integral-inequalities} generalizes this framework to optimization problems in which parametrized variational inequalities appear as constraints.
Further relaxations into SOS programs in the case of polynomial data are described in \cref{sec:SOS}.

\subsection{Integral variational problems}
\label{ss:variational-problems}

Consider the integral variational problem~\cref{eq:primal-vp}, where $\Omega$ is a bounded open Lipschitz domain. For simplicity we assume that the domain boundary is specified as a finite union, $\partial\Omega = \cup_{i=1}^s \partial\Omega_i$, where each $\partial\Omega_i$ is a smooth compact manifold.
We assume the integrand function $f$ in~\cref{eq:primal-vp} is continuous and gives a well defined integral over $\Omega$ for all $u\in W^{1,p}$ satisfying the constraints. For instance, in the absence of pointwise constraints that would ensure $u(x)$ and $\nabla u(x)$ are uniformly bounded, we assume that $|f(x,y,z)|$ grows no faster than $|y|^p + |z|^p$ as $|y|,|z|\to\infty$.
Such assumptions guarantee that the infimum in~\cref{eq:primal-vp} is well defined but \emph{not} that it is finite or attained. Further assumptions to guarantee the existence of minimizers, such as coercivity of $f$, are not needed---our approach bounds global infima below whether or not they are attained. More general growth conditions and classical results on the existence of minimizers can be found in~\cite{Dacorogna2008}.

For the purpose of illustration throughout this section, we suppose the constraints on $u$ in~\cref{eq:primal-vp} include one integral inequality, one pointwise equality on $\Omega$, and one equality on each boundary component $\partial\Omega_i$ (i.e, one boundary condition):
\begin{subequations}
	\label{eq:constraints-all}
	\begin{align}
	\label{int-both}
	\int_\Omega a(x,u,\nabla u) \,\dx &\ge 0,
	\\
	\label{pw-domain}
	c(x,u,\nabla u)&= 0 \quad x\in\Omega,
	\\
	\label{pw-boundary}
	d_i(x,u)&= 0 \quad x\in\partial\Omega_i,\quad i=1,\ldots,s.
	\end{align}
\end{subequations}
We assume that the functions $a,c,d_i$ are continuous and that $a(x,u,\nabla u)$ is integrable for all $u\in W^{1,p}$.
It is straightforward to include integral constraints that are equalities, pointwise constraints that are inequalities, and multiple constraints of each type, but~\cref{eq:constraints-all} suffices to explain how all such constraints can be incorporated.

\subsubsection{\label{sec:lagrange-dual}A Lagrange dual problem}

Lagrangian duality gives a standard way to pose a maximization problem whose supremum $\cL^*$ is a lower bound on the infimum $\cF^*$ of~\cref{eq:primal-vp}. Specifically, $\cF^*$ can first be written as a minmax problem wherein constraints are enforced by Lagrange multipliers, and then bounded below by the corresponding maxmin problem. In what follows we keep pointwise constraints explicit and introduce Lagrange multipliers only for integral constraints like~\cref{int-both}, so for the variational problem~\cref{eq:primal-vp} with constraints~\cref{eq:constraints-all} the Lagrange dual problem is derived by
\begin{subequations}
	\begin{align}
	\cF^* &= \inf_{\substack{u\in W^{1,p}\\ (\ref{eq:constraints-all}b,c)}}\sup_{\eta \in \mathbb{R}_+,}\int_\Omega[f-\eta a](x,u,\nabla u)\,\dx \label{eq:F lagrangian} \\
	 &\ge \sup_{\eta \in \mathbb{R}_+} \inf_{\substack{u\in W^{1,p}\\ (\ref{eq:constraints-all}b,c)}} \int_\Omega[f-\eta a](x,u,\nabla u)\,\dx =: \cL^*. \label{eq:L lagrangian}
	\end{align}
\end{subequations}
The scalar Lagrange multiplier $\eta$ must be in the nonnegative set $\R_+$ because \cref{int-both} is an integral nonnegativity constraint. (For an integral \emph{equality} constraint, the Lagrange multiplier can take any value in $\R$.) The inequality $\cF^*\ge\cL^*$ expresses weak duality; strong duality is said to hold if $\cF^*=\cL^*$.

Computing $\cL^*$ is hard because the inner minimization problem in~\cref{eq:L lagrangian}, like the minimization defining $\cF^*$, is generally non-convex and intractable. We therefore relax~\cref{eq:L lagrangian} to obtain more tractable maximizations giving lower bounds on $\cL^*$ that may or may not be sharp. \Cref{sec:pointwise-dual} explains the relaxations leading to the PDR, and \cref{sec:SOS} explains the further relaxation to SOS programs in the case of polynomial data.
Although additional Lagrange multipliers could be introduced for pointwise constraints such as~(\ref{eq:constraints-all}b,c) to formulate a Lagrange dual problem in a more standard form than~\cref{eq:L lagrangian}, the SOS programs ultimately obtained would be the same as those in~\cref{sec:SOS}.

\subsubsection{\label{sec:pointwise-dual}Pointwise dual relaxation}

Let $v:\Omega \to \mathbb{R}^{m\times n}$ denote the weak gradient $\nabla u$ and write $\cF^*$ with the definition of $v$ included as a constraint,
\beq
\cF^* = \inf_{\substack{u\in W^{1,p}\\v\in L^p}}\int_\Omega f(x,u,v)\,\dx  \;\;\text{ s.t.}\!\!
\begin{array}[t]{rll}
&\int_\Omega a(x,u,v) \,\dx \ge 0, &\text{(a)}\\[0.5ex]
&v = \nabla u \text{ weakly on }\Omega, &\text{(b)}\\[0.5ex]
&c(x,u,v)= 0 \text{ on }\Omega, &\text{(c)}\\[0.5ex]
&d_i(x,u)= 0 \text{ on }\partial\Omega_i,~ i=1,\ldots,s. &\text{(d)}
\end{array}
\label{eq:F-div-dual}
\eeq
To derive the PDR, we first formulate a dual maxmin problem that is similar to the Lagrange dual problem~\cref{eq:L lagrangian} but includes a Lagrange-multiplier-like function used to enforce $v=\nabla u$. Then, we bound the integral objective of this maxmin problem from below by the pointwise infimum of its integrand.

The constraint $v=\nabla u$ is imposed using the divergence theorem for vector-valued functions $x \mapsto \vphi(x,u(x))$, where $\varphi: \overline{\Omega}\times\mathbb{R}^m \to \mathbb{R}^n$ is continuously differentiable:
\beq
\label{eq:div-theorem}
\int_\Omega\diver_x\vphi(x,u(x))\,\dx - \sum_{i=1}^s \int_{\partial\Omega_i}\varphi(x,u)\cdot n_i(x) \,\dbound = 0.
\eeq
Here, $n_i(x)$ is the outward unit normal vector on $\partial\Omega_i$. Let $\Phi$ denote the set of $\vphi$ for which the above identity holds and for which the map $x \mapsto \varphi(x,u(x))$ is differentiable via the chain rule for all $u\in W^{1,p}(\Omega,\R^m)$. (See~\cite{Fantuzzi2022} for a more explicit characterization of $\Phi$.) Since $\cD\vphi$ in~\cref{eq: D} is defined in order to satisfy $\diver_x\vphi(x,u(x))=\cD\vphi(x,u,\nabla u)$, the left-hand integrand in~\cref{eq:div-theorem} can be replaced by $\cD\vphi(x,u,\nabla u)$. If $v=\nabla u$ weakly, therefore, identity~\cref{eq:div-theorem} holds for all $\vphi\in\Phi$ with the left-hand integrand replaced by $\cD\vphi(x,u,v)$. The proof of~\cite[Lemma~1]{Korda2018} implies that the converse is also true, so the divergence theorem characterizes the weak gradient $\nabla u$. The next Lemma summarizes these observations.
\begin{lemma}
\label{th:weak-derivative-lemma}
The functions $u \in W^{1,p}(\Omega;\mathbb{R}^m)$ and $v\in L^p(\Omega;\mathbb{R}^{m \times n})$ satisfy $v=\nabla u$ weakly on $\Omega$ if and only if
\begin{equation}
\label{eq:div-theorem-2}
\int_\Omega\mathcal{D}\varphi(x,u,v) \,\dx - \sum_{i=1}^s \int_{\partial\Omega_i}\varphi(x,u)\cdot n_i(x) \,\dbound = 0 \quad \forall \vphi \in \Phi,
\end{equation}
where $\cD\vphi$ is defined by~\cref{eq: D} and $\Phi$ is the subset of $C^1(\overline{\Omega}\times\R^m;\R^n)$
such that the divergence theorem~\cref{eq:div-theorem} holds for all $u\in W^{1,p}$.
\end{lemma}

\Cref{th:weak-derivative-lemma} lets us rewrite~\cref{eq:F-div-dual} as a minmax problem in which $\vphi$ is a Lagrange-multiplier-like function enforcing $v=\nabla u$. As in \cref{sec:lagrange-dual}, we also introduce a Lagrange multiplier $\eta\in\R_+$ for the integral constraint~(\ref{eq:F-div-dual}a)
but keep the pointwise constraints explicit: 
\beq
\cF^*= \!\!\inf_{\substack{u\in W^{1,p}\\v\in L^p\\ \text{s.t.}\,\text{(\ref{eq:F-div-dual}c,d)}}}
	\sup_{\substack{\eta \in \mathbb{R}_+\\\varphi \in \Phi}}
	\left\{\!\int_\Omega \![f-\eta a + \mathcal{D}\vphi](x,u,v) \dx -\sum_{i=1}^s\!\int_{\partial\Omega_i}\!\!\varphi(x,u)\cdot n_i(x) \dbound \right\}.
\label{eq:F-div-dual-2}
\eeq
The fact that~\cref{eq:F-div-dual-2} is equal to~\cref{eq:F lagrangian} follows from the observation that the supremum of the left-hand side of~\cref{eq:div-theorem-2} over $\vphi\in\Phi$ is zero if $v=\nabla u$ weakly but is $\infty$ otherwise. Often, it is possible to constrain $\vphi$ so that the boundary integrals in~\cref{eq:F-div-dual-2} vanish for all $u\in W^{1,p}$ that satisfy the boundary conditions $d_i(x,u)=0$. In such cases, the dual relaxations derived below have no boundary terms. This boundary-free version was sketched in \cref{sec: sketch} and is formulated precisely in \cref{sec:no-boundary} below.

Exchanging the supremum over $\eta$ with the infima in~\cref{eq:F-div-dual-2} gives an expression equivalent to the Lagrange dual problem defined in~\cref{eq:L lagrangian}. Exchanging the suprema over both $\eta$ and $\vphi$ with the infima in~\cref{eq:F-div-dual-2}, therefore, gives a maxmin problem that bounds the optimal value $\cL^*$ of the Lagrange dual problem from below, after which taking pointwise infima of the continuous integrands produces a further lower bound:\small
\begin{subequations}
\begin{align}
\tag{\normalsize\theequation a}
\cL^*&\ge
\sup_{\substack{\eta \in \mathbb{R}_+\\\varphi \in \Phi}}
	\inf_{\substack{u\in W^{1,p}\\v\in L^p\\ \text{(\ref{eq:F-div-dual}c,d)}}}\!\!
	\left\{\!\int\limits_\Omega \![f-\eta a + \mathcal{D}\vphi](x,u,v)\,\dx -\sum_{i=1}^s \int\limits_{\partial\Omega_i}\!\!\varphi(x,u)\cdot n_i(x)\,\dbound \right\}
\label{eq:div-dual}
\\
\tag{\normalsize\theequation b}
&\ge
\sup_{\substack{\eta \in \mathbb{R}_+\\\varphi \in \Phi}}
\left\{\! \int\limits_{\Omega}\!\!\!\! \inf_{\substack{ y\in \R^m \\ z \in \R^{m\times n}\\c(x,y,z)=0}}\!\!\![f-\eta a + \mathcal{D}\vphi](x,y,z)\, \dx+ \sum_{i=1}^s \int\limits_{\partial\Omega_i}\!\! \inf_{\substack{y\in\R^m\\d_i(x,y)=0}}\!\!\!\!-\varphi(x,y)\cdot n_i(x)\dbound\right\}.
\label{eq:relaxed-dual}
\end{align}
\end{subequations}

\normalsize
Finally, we bound the pointwise infima in~\cref{eq:relaxed-dual} from below by the largest continuous functions $h\in C(\overline{\Omega},\R)$ and $\ell_i\in C(\partial\Omega_i,\R)$ satisfying
\begin{subequations}
	\label{eq:relaxed-inequalities-all}
	\begin{align}
	\label{eq:pw-inequality-bulk}
	f(x,y,z) - \eta\, a(x,y,z) + \mathcal{D}\varphi(x,y,z)  - h(x) &\geq 0 \; \text { on } \Gamma,\\
	-\varphi(x,y)\cdot n_i(x) - \ell_i(x) &\geq 0  \; \text { on } \Lambda_i, \quad i=1,\ldots,s,
	\label{eq:pw-inequality-boundary}
	\end{align}
\end{subequations}
where $\Gamma$ and $\Lambda_i$ are the sets on which the pointwise constraints are satisfied:
\begin{subequations}
\label{eq:Gamma-Lambda}
\begin{align}
\label{eq:Gamma-def}
\Gamma &:= \{(x,y,z) \in \Omega \times \mathbb{R}^m \times \mathbb{R}^{m \times n}\colon\; c(x,y,z)= 0\},\\
\Lambda_i &:= \{(x,y) \in \partial\Omega_i \times \mathbb{R}^m \colon \; d_i(x,y)= 0\}. \label{eq:Lambda-def}
\end{align}
\end{subequations}
This yields
\begin{equation}
\label{eq:pdr-2}
\cF^* \geq \cL^* \geq \Lpdr :=
\sup_{\substack{ \eta \in \mathbb{R}_+,\,\varphi \in \Phi,\\ h \in C(\overline{\Omega}),\, \ell_i \in C(\partial\Omega_i)\\ \text{\rm ~(\ref{eq:relaxed-inequalities-all}a,b)} }} \;
\left\{
\int_\Omega h(x) \,\dx
+\sum_{i=1}^s \int_{\partial\Omega_i} \ell_i(x) \,\dbound
\right\},
\end{equation}
which are some of the inequalities asserted in~\cref{eq:duality-relation-of-bounds}. The maximization problem in~\cref{eq:pdr-2} is what we call the \textit{pointwise dual relaxation} (PDR).

The pointwise estimates used to bound the right-hand side of \cref{eq:div-dual} below by \cref{eq:relaxed-dual} can be very conservative for suboptimal $\vphi$, often giving pointwise infima of $-\infty$ at some or all $x$ values. In such cases there fail to exist continuous $h$ or $\ell_i$ satisfying~\cref{eq:relaxed-inequalities-all}. When $\vphi$ is optimized along with $h$ and $\ell_i$, however, the estimates can be sharp: the equality $\cF^*=\Lpdr$ is proved for three classes of problems in \cref{sec: sharp} and for some other cases in \cite{Fantuzzi2022}. In general, it is an open challenge to characterize when the inequality $\cF^*\ge\Lpdr$ is an equality. Nonetheless, for any particular variational problem where one seeks an \emph{explicit} lower bound on the global infimum $\cF^*$, the PDR remains useful because often $\Lpdr$ can be bounded from below or, sometimes, computed exactly.

The reason $\Lpdr$ gives an explicit lower bound on $\cF^*$ more easily than either $\cL^*$ or the right-hand side of~\cref{eq:div-dual} is that, while all three maximizations are convex, the PDR does not require an inner minimization over functions $u\in W^{1,p}$; such minimizations are generally non-convex and intractable. Solving the PDR exactly is also typically impossible, but the lack of an inner minimization makes it easier to derive explicit lower bounds on $\Lpdr$. One can even hope to optimize this lower bound over the scalar $\eta$ and the functions $\vphi$, $h$, $\ell_1$, $\ldots$, $\ell_s$, at least when the search for such functions is restricted to finite-dimensional function spaces. \Cref{sec:SOS} describes how this can be done computationally by solving SOS programs when~\cref{eq:primal-vp} has polynomial data.

\subsubsection{\label{sec:no-boundary}Pointwise dual relaxation without boundary terms}

For certain $\vphi$, the integral $\int_\Omega\diver_x\vphi(x,u(x))\,\dx$ vanishes for all $u$ satisfying the problem constraints, and then the boundary integrals in~\cref{eq:div-theorem} can be omitted as in the sketch of \cref{sec: sketch}. For instance, if $u$ vanishes at the boundaries and $\vphi(x,y)=|y|^2F(x)$ for sufficiently regular $F:\Omega\to\R^n$, then the boundary integrals in~\cref{eq:div-theorem} vanish. This is true even when $F(x)$ is singular on $\partial\Omega$, provided the singularity is weak enough, in which case the method is well defined only without boundary integrals.

The requirement that $\int_\Omega\diver_x\vphi(x,u(x))\,\dx=0$ for all admissible $u$ generally excludes some $\vphi\in\Phi$, but it may also include some $\vphi\notin\Phi$ whose restriction to the boundary of $\Omega$ are not well defined. For any given constraints on $u$, we denote by $\widehat\Phi$ the set of  continuously differentiable $\vphi: \Omega \times \mathbb{R}^m \to \mathbb{R}^n$ for which $\int_\Omega\diver_x\vphi(x,u(x))\,\dx=0$, but which need not be well-defined on $\partial\Omega$. The dependence of the set $\widehat\Phi$ on boundary conditions and other constraints is in contrast to the set $\Phi$, whose elements only need to satisfy the divergence theorem identity~\cref{eq:div-theorem} for all $u\in W^{1,p}$. Optimizing over $\vphi\in\widehat\Phi$ instead of $\Phi$ leads to a version of the PDR~\cref{eq:pdr-2} without boundary terms:
\begin{equation}
\label{eq:pdr-no-boundary}
\cF^* \ge \cL^*  \ge \HLpdr :=
\sup_{\substack{ \eta \in \mathbb{R}_+,\,\varphi \in \widehat\Phi, \\  h \in C(\overline{\Omega}) \text{~s.t.~\cref{eq:pw-inequality-bulk}} }} \;
\int_\Omega h(x) \,\dx.
\end{equation}
Since neither the set $\Phi$ nor the set $\widehat\Phi$ contains the other, it is not clear which of $\Lpdr$ and $\HLpdr$ is the better lower bound on $\cF^*$. However, $\cF^*=\Lpdr=\HLpdr$ for the classes of problems considered in~\cref{sec: sharp}, which include the following example.

\begin{example}
\label{ex:simple-variational-problem}
Consider the variational problem mentioned in \cref{sec: sketch},
\beq
\label{eq: intro ex 1}
\mathcal{F}^*=\inf_{\substack{u \in W^{1,2}\\u(\pm \frac{\pi}{3})=0}}\int_{-\pi/3}^{\pi/3}(u_x^2-u^2-2u)\dx,
\eeq
For this example, the Euler--Lagrange equation $u_{xx}+u+1=0$ is necessary and sufficient for global optimality~\cite{Evans1998}, giving the minimizer $u^*(x)=2\cos(x)-1$ and the minimum $\mathcal{F}^*= 2(\tfrac{\pi}{3}-\sqrt{3})$. The PDR approach yields an exact lower bound on $\cF^*$, either by finding an optimal $\vphi\in\widehat\Phi$ yielding $\HLpdr\ge2(\tfrac{\pi}{3}-\sqrt{3})$, or by finding an optimizing sequence of $\vphi\in\Phi$ yielding the same lower bound on $\Lpdr$.

First we consider the boundary-free formulation~\cref{eq:pdr-no-boundary}, here with no multiplier $\eta$ because there is no integral constraint. The maximum $\HLpdr$ is attained with
\begin{align}
\vphi^*(x,y)&=
\left(\frac{2 \sin x}{2\cos x - 1}\right)y^2
&\text{and}&&
h^*(x) &= 1-2\cos x.
\label{eq:first-ex-phi}
\end{align}
These expressions are a special case of the construction in \cref{sec:quadratic-minimization} of optimal $\vphi^*$ in terms of $u^*$. In general, optimal $\vphi^*$ are not known explicitly, but one can construct near-optimal $\vphi$ computationally, as done for the present problem in \cref{sec:simple-variational-problem-sos}. The coefficient of $y^2$ in~\cref{eq:first-ex-phi} has a degree-one singularity as $x$ approaches the boundary points $\pm\pi/3$, but still the integral of $\cD\vphi(x,u,u_x)$ over $\Omega=(-\pi/3,\pi/3)$ vanishes for all $u\in W^{1,2}$ that go to zero at the boundaries. This $\vphi^*$, therefore, belongs to the admissible set $\widehat\Phi$ for the boundary-free formulation~\cref{eq:pdr-no-boundary}, but it does not belong to $\Phi$ since the boundary inequality~\cref{eq:pw-inequality-boundary} is not defined. With the chosen $\vphi^*$ and $h^*$,
	\beq
	[f+\cD\vphi^*-h^*](x,y,z)= \left[z + \left(\frac{2 \sin x}{2\cos x - 1}\right)y\right]^2
	    + \frac{\left( y - 2\cos x + 1\right)^2}{2\cos x - 1}.
	\eeq
This quantity is manifestly nonnegative for all $(x,y,z) \in \Gamma = \Omega \times \R \times \R$. Integrating $h^*$ over $\Omega$ gives $\HLpdr\ge2(\tfrac{\pi}{3}-\sqrt{3}) = \cF^*$, showing that $\cF^*=\HLpdr$ for this example. 

The formulation~\cref{eq:pdr-2} that includes boundary terms can give the same sharp lower bound on $\cF^*$, but we cannot use the function~$\vphi^*$ in \cref{eq:first-ex-phi} because it is singular on the boundary. Instead, we can move the singularities outside $\Omega$ by choosing
	\begin{align}
	\vphi(x,y) &=
	\left(\frac{2 \sin x}{2 \cos x - \omega}\right)y^2 
	&\text{and}&&
	h(x) &= 1 - \frac{2 \cos x}{\omega}
	\label{eq:first-ex-phi-2}
	\end{align}
with parameter $\omega\in(0,1)$. These functions are smooth on $\overline{\Omega}$ and approach the choices in~\cref{eq:first-ex-phi} as $\omega \nearrow 1$. In this limit these parametrized $\vphi$ and $h$, along with the corresponding choices $\ell_1=\ell_2=0$, form an optimizing sequence for~\cref{eq:pdr-2} that shows $\Lpdr\ge2(\tfrac{\pi}{3}-\sqrt{3})=\cF^*$, and so $\cF^*=\Lpdr$.
\end{example}

\subsection{Optimization subject to parametrized integral inequalities}
\label{ss:integral-inequalities}

The ideas of the PDR in~\cref{ss:variational-problems} can be used
to relax optimization problems for a parameter $\lambda\in\R^\tau$ in which parametrized integral inequalities $\cF^*(\lambda)\ge0$ appear as \emph{constraints}. In particular, we consider convex problems in the form
\begin{equation}
\label{eq:integral-inequality-setup}
\cB^* = \sup_{\lambda \in \R^\tau} \; b(\lambda) \quad
\text{s.t.} \quad
\mathcal{F}^*(\lambda):= \inf_{ \substack{u \in W^{1,p}\\\text{\rm s.t.~(\ref{eq:constraints-all}a--c)}}}
\int_{\Omega} f(x,u,\nabla u,\lambda) \,\dx \ge 0,
\end{equation}
where $b:\R^\tau \to \mathbb{R}$ is a concave cost function, and the integrand $f$ is as in \cref{ss:variational-problems} but also depends affinely on $\lambda$. The conditions~\cref{int-both,pw-domain,pw-boundary} defining admissible $u$ are still independent of $\lambda$. 
Problems of the form~\cref{eq:integral-inequality-setup} generalize the variational problems of \cref{ss:variational-problems}, which can be put in the form~\cref{eq:integral-inequality-setup} by letting $\tau=1$, $b(\lambda)=\lambda$, and $\cF^*(\lambda) = \inf \int_{\Omega} f(x,u,\nabla u) \dx - \lambda$. It is straightforward to further generalize~\cref{eq:integral-inequality-setup} with more than one variational constraint.

Applying the PDR approach of \cref{sec:pointwise-dual} to $\cF^*(\lambda)$ for fixed $\lambda$ gives a lower bound $\Lpdr(\lambda)$. Replacing the constraint $\cF^*(\lambda)\ge0$ in~\cref{eq:integral-inequality-setup} with the sufficient condition $\Lpdr(\lambda)\ge0$ gives 
\begin{equation}
\label{eq:integral-inequality-relaxation}
\cB^*\ge \Bpdr:= \sup_{\lambda \in \R^\tau} \; b(\lambda) \quad
\text{s.t.} \quad \Lpdr(\lambda) \ge 0.
\end{equation}
Alternatively, the boundary-free relaxation $\HLpdr$ described in \cref{sec:no-boundary} can be used in place of $\Lpdr$ in~\cref{eq:integral-inequality-relaxation}, giving a different lower bound $\cB^*\ge\HBpdr$. Using expressions~\cref{eq:pdr-2} and~\cref{eq:pdr-no-boundary} for $\Lpdr$ and $\HLpdr$ lets us write
\begin{subequations}
	\begin{align}
	\label{eq:Bpdr}
	\Bpdr&= \!\!\!\!
	\sup_{\substack{\lambda \in \R^\tau,\,\eta \in\R_+,\\ \varphi \in \Phi,\,h\in C(\overline{\Omega}),\\ \ell_i\in C(\partial\Omega_i)}}
	\!\!
	b(\lambda)\;
	\text{ s.t. (\ref{eq:relaxed-inequalities-all}a,b), }
	\displaystyle \int_\Omega h(x)\,\dx + \sum_{i=1}^s \int_{\partial\Omega_i} \ell_i(x)\,\dbound \ge 0, \\
	\label{eq:Bpdr-hat}
	\HBpdr&=
	\sup_{\substack{\lambda \in \R^\tau,\,\eta \in\R_+,\\ \varphi \in \widehat\Phi,\,h\in C(\overline{\Omega})}}
	b(\lambda)\;
	\text{ s.t. (\ref{eq:relaxed-inequalities-all}a), }
	\displaystyle \int_\Omega h(x)\,\dx \ge 0.
	\end{align}
\end{subequations}

Problems~\cref{eq:Bpdr,eq:Bpdr-hat} are convex because the cost function $b$ is concave and the optimization variables $\lambda,\eta,\vphi,h,\ell_i$ enter the constraints only linearly. These maximizations are exactly solvable in some simple cases, such as \cref{ex:poincare1d-simple-example} below, but not in general. Nonetheless, if one can find any suboptimal $\lambda,\eta,\vphi,h,\ell_i$ for~\cref{eq:Bpdr} or~\cref{eq:Bpdr-hat}, then evaluating $b(\lambda)$ gives a lower bound on $\cB^*$. Checking the constraints of either problem requires verifying pointwise inequalities on finite-dimensional sets, whereas checking the constraints of~\cref{eq:integral-inequality-setup} requires verifying integral inequalities for infinite-dimensional sets of $u$, which generally is much harder. The pointwise inequalities can also be hard to verify but, in the case of polynomial data,~\cref{eq:Bpdr,eq:Bpdr-hat} can be further relaxed to tractable SOS programs (see \cref{ss:sos-int-ineq}).

\begin{example}
\label{ex:poincare1d-simple-example}
The Poincar\'e inequality for functions $u\in W^{1,2}([-1,1])$ subject to vanishing Dirichlet boundary conditions states that $\lambda \Vert u\Vert_{L^2}^2\le \Vert u_x\Vert_{L^2}^2$  for some constant $\lambda > 0$. The largest such $\lambda$ is the solution of
\begin{equation}
\lambda^* = \sup\, \lambda \quad \text{s.t.} \quad
\int_{-1}^1 \left(u_x^2 - \lambda u^2\right) \,\dx \ge 0 \quad \forall u\in W^{1,2}: u(\pm 1)=0.
\end{equation}
The value $\lambda^* = \pi^2/4$ can be found by solving the Euler--Lagrange equation $-u_{xx}=\lambda u$, whose leading eigenfunction $u^*=\cos(\pi x/2)$ saturates the Poincar\'e inequality. 
Alternatively, the PDR approach can produce an exact lower bound on $\lambda^*$. In the present example, the relaxation~\cref{eq:Bpdr-hat} gives $\lambda^* \geq \HBpdr$ with
\begin{equation}
\label{eq:poincare-pdr}
\HBpdr = \!\!
\sup_{\substack{\lambda \in \R, \,\varphi \in \widehat\Phi\\ h\in C(\overline{\Omega})}}  \; \lambda
\;\; \text{ s.t.}
\begin{array}[t]{l}
z^2 - \lambda y^2 + \mathcal{D}\vphi(x,y,z) \ge h(x) \text{ on }(-1,1)\times \R \times \R,\\[0.5ex]
\int_\Omega h(x)\,\dx \ge 0,
\end{array}
\end{equation}
where $\mathcal{D}\vphi(x,y,z)=\vphi_x(x,y)+\vphi_y(x,y)z$.
The maximum $\HBpdr$ is attained with the optimizers $h^*= 0$ and $\vphi^*(x,y) = (\pi/2)\tan(\pi x/2) y^2$. Note that $\vphi^*=-(u^*_x/u^*)y^2$ in terms of the Euler--Lagrange solution $u^*(x)$, as suggested by the $\vphi$ constructed in~\cref{ss:elliptic-eigenvalues} for a generalization of the present example.\footnote{{This dual formulation for the optimal Poincar\'e constant is closely related to that in~\cite{Hersch1961} and related works, where maximizers that take the form $\nabla u^*/u^*$.}} This $\vphi^*$ is singular at the boundary points, but it still belongs to $\widehat\Phi$ because the integral of $\cD\vphi(x,u,u_x)$ over the domain vanishes for all $u\in W^{1,2}$ that go to zero at the boundaries. The second constraint in~\cref{eq:poincare-pdr} is trivially satisfied, and the first is satisfied for any $\lambda\le \pi^2/4$; with $\lambda=\pi^2/4$, the expression that must be nonnegative is equal to $[z+(\pi/2)\tan(\pi x/2)y]^2$. This gives the lower bound $\HBpdr\ge\pi^2/4$, which is the exact value of $\cB^*$ and $\HBpdr$ in this example. It is also the value of $\Bpdr$; an optimizing sequence for~\cref{eq:Bpdr-hat} showing $\Bpdr\ge\pi^2/4$ is $h,\ell_1,\ell_2=0$ and $\varphi=\sqrt\lambda\tan(\sqrt\lambda x)y^2$ with $\lambda\nearrow\pi^2/4$.
\end{example}

\begin{remark}
\label{rem:null}
The PDR~\cref{eq:pdr-2} makes use of the fact that the volume integral of $\cD\vphi(x,u,\nabla u)$ can be expressed as a boundary integral, meaning $\cD\vphi(x,u,\nabla u)$ is a \textit{null Lagrangian}. All null Lagrangian functions of $(x,u,\nabla u)$ can be generated from $\vphi(x,u)$ in this way when $m=1$ or $n=1$. With $m,n\ge2$, generalizing the PDR to include all possible null Lagrangians requires letting $\vphi$ have certain polynomial dependence on $\nabla u$ as well~\cite{Edelen1962}. This generalization is left for future work.
\end{remark}

\section{Relaxation to SOS programs for polynomial data}
\label{sec:SOS}

We now consider the particular case in which the integral variational problem~\cref{eq:primal-vp} or the integral-constrained optimization problem~\cref{eq:integral-inequality-setup} has polynomial data, meaning that all functions appearing as integrands or in pointwise constraints are polynomial in $(x,u,\nabla u)$, and that the domain $\Omega$ can be specified by polynomials in $x$. Precisely, we assume that $\overline\Omega$ is a basic closed semialgebraic set defined by $s$ polynomial inequalities,
\begin{equation}
\overline\Omega = \{x \in \R^n:\; g_1(x)\ge 0,\,\ldots,\,g_s(x)\ge 0\}.
\label{def-omega}
\end{equation}
We also assume that each $g_i$ vanishes identically on a single boundary component, so
\begin{align}
\partial\Omega_i
&= \{x \in \R^n:\;g_i(x)=0 \text{ and } g_j(x)\geq 0 \text{ for } j\neq i \},
\end{align}
and that $\nabla g_i$ does not vanish anywhere on $\partial\Omega_i$, so the outward unit normal vector is well defined as $n_i(x) = -\nabla g_i/|\nabla g_i|$. Then, the sets $\Gamma$ and $\Lambda_i$ defined in~\cref{eq:Gamma-def,eq:Lambda-def} admit the semialgebraic definitions
\begin{subequations}
\label{eq:Gamma-Lambda-semialg}
\begin{align}
\Gamma &= \{(x,y,z) \in \R^n \times \R^m \times \R^{m\times n}:
c(x,y,z)= 0,\,g_1(x)\ge 0,\ldots,g_s(x)\geq 0\},
\\
\Lambda_i &= \{(x,y) \in \R^n \times \R^m:\,
d_i(x,y) = 0,\,g_i(x)=0,\,
g_j(x)\geq 0 \text{ for } j\neq i
\}.
\end{align}
\end{subequations}
Under these assumptions, the PDRs $\Lpdr$ in~\cref{eq:pdr-2}, $\HLpdr$ in~\cref{eq:pdr-no-boundary}, $\Bpdr$ in \cref{eq:Bpdr}, and $\HBpdr$ in~\cref{eq:Bpdr-hat} can be further relaxed into finite-dimensional SOS programs that are computationally tractable and give lower bounds on the exact PDR values.

\subsection{\label{ss:sos-variational}Integral variational problems}

For any positive integer $\nu$, we can bound $\Lpdr$ from below by restricting the maximization in~\cref{eq:pdr-2} to 
\begin{align}
\vphi &\in \Phi\cap\R^n[x,y]_\nu, &
h &\in \R[x]_\nu, &
\ell_i &= q_i |\nabla g_i| \; \text{ with } \; q_i \in \R[x]_\nu,
\end{align}
where $\R^n[\cdot]_\nu$ is the space of $n$-dimensional vectors of polynomials in the bracketed variables whose degrees are at most $\nu$. More generally, one can choose any finite vector space of polynomials, 
or use rational functions with fixed nonnegative denominators (see \cref{rem:rational} below). The intersection with $\Phi$
excludes polynomials $\vphi(x,y)$ whose degree in $y$ is too large to guarantee integrability of $\cD\vphi(x,u,\nabla u)$ for all $u\in W^{1,p}$. Recalling that $n_i(x) = -\nabla g_i/|\nabla g_i|$ for each $i=1,\ldots,s$, we find
\begin{equation}
\label{eq:vp-finite-dimensional}
\Lpdr\ge
\Lnu :=\!\!\!\!\!\!\!
\sup_{ \substack{\vphi \in \Phi \cap \R^n[x,y]_\nu \\h, q_i \in \R[x]_\nu, \, \eta \in \R_+} }
\!\!\!\!\!\!L(h,q_1,\ldots,q_s)
\;\; \text{s.t.}
\begin{array}[t]{l}
f - \eta a + \mathcal{D}\varphi - h \geq 0  \,\text { on } \Gamma,\\
\varphi\cdot \nabla g_i - q_i|\nabla g_i|^2 \geq 0  \, \text { on } \Lambda_i, 
\end{array}
\end{equation}
where the maximization objective is
\beq
\label{eq:L-objective}
L(h,q_1,\ldots,q_s) : = \int_{\Omega} h(x) \, \dx + \sum_{i=1}^s \int_{\partial\Omega_i} q_i(x)|\nabla g_i(x)| \,\dbound.
\eeq
The tunable polynomials $h$ and $q_i$ appear linearly in $L$, so if they are expressed in any chosen basis with tunable coefficients, then evaluating the integrals in~\cref{eq:L-objective} for each polynomial basis function gives $L$ as an explicit linear function of the tunable coefficients. These integrals can be evaluated explicitly in some cases---e.g., when $\Omega$ is a polyhedron, a ball, or an intersection of such sets. Otherwise, they can be evaluated numerically to find the numbers that multiply the tunable coefficients in $L$.

Problem~\cref{eq:vp-finite-dimensional} is finite-dimensional but still intractable because pointwise polynomial inequalities are NP-hard to verify in general~\cite{Murty1987}. For this reason, we follow a standard approach~\cite{Parrilo2003,Lasserre2001} and strengthen all polynomial nonnegativity constraints into weighted SOS constraints. Letting $\Sigma[\cdot]$ denote the subset of nonnegative polynomials in the bracketed variables that can be written as sums of squares of other polynomials, we consider the sets
\begin{subequations}
\label{eq:quad-mods}
	\begin{align}
	\mathcal{Q}(\Gamma) &:= \bigg\{\rho c + \sigma_0 + \sum_{i=1}^s g_i \sigma_i:\; \sigma_k \in \Sigma[x,y,z], \,\rho \in \R[x,y,z] \bigg\}, \\
	\mathcal{Q}(\Lambda_i) &:= \bigg\{\rho_1 d_i + \rho_2 g_i  + \sigma_0 + \sum_{j\neq i} g_i \sigma_i:\; \sigma_k \in \Sigma[x,y], \,\rho_j \in \R[x,y] \bigg\}.
	\end{align}
\end{subequations}
By design, every polynomial in $\cQ(\Gamma)$ is nonnegative on $\Gamma$ since $c=0$ and $g_i\ge0$ on $\Gamma$, and $\sigma_i\ge0$. Similarly, every polynomial in $\cQ(\Lambda_i)$ is nonnegative on $\Lambda_i$. (The converse statements are not always true; see~\cite[Example~5.6]{Nie2013} for a counterexample.) We let $\cQ_\nu$ denote the restriction of these polynomial sets to degrees no larger than $\nu$:
\begin{align}
\label{eq:quad-mods-nu}
	\mathcal{Q}_\nu(\Gamma) &= \mathcal{Q}(\Gamma) \cap \R[x,y,z]_\nu, &
	\mathcal{Q}_\nu(\Lambda_i) &= \mathcal{Q}(\Lambda_i) \cap \R[x,y]_\nu.
\end{align}
Requiring a polynomial to belong to a finite-degree set such as $\cQ_\nu(\Gamma)$ or $\cQ_\nu(\Lambda_i)$ is a \emph{weighted SOS} condition, which can be included as a constraint in an SOS program.

The constraints of~\cref{eq:vp-finite-dimensional}, which enforce polynomial nonnegativity on semialgebraic sets, can be strengthened into weighted SOS constraints to obtain
\begin{equation}
\label{eq:vp-SOS}
\Lnu\ge\Lsos :=\!\!\!\!\!\!\!
\sup_{ \substack{\vphi \in \Phi \cap \R^n[x,y]_\nu\\h, q_i \in \R[x]_\nu, \, \eta \in \R_+} }
\!\!\!\!\!\!L(h,q_1,\ldots,q_s)
\;\; \text{s.t.}
\begin{array}[t]{l}
f - \eta a + \mathcal{D}\varphi - h \in \cQ_\nu(\Gamma),\\
\varphi\cdot \nabla g_i - q_i|\nabla g_i|^2 \in \cQ_\nu(\Lambda_i). 
\end{array}
\end{equation}
This is an SOS program, provided $L(h,q_1,\ldots,q_s)$ is expressed as an explicit linear function of the tunable coefficients of $h$ and $q_i$ as described after~\cref{eq:L-objective}. We likewise define $\HLsos$ as the SOS program that relaxes the boundary-free formulation $\HLpdr$,
\begin{equation}
\label{eq:L-sos-boundary-free}
\HLsos :=
\sup_{ \substack{\vphi \in \widehat\Phi \cap \R^n[x,y]_\nu,\\h \in \R[x]_\nu, \, \eta \in \R_+} }
L(h) \;\;\; \text{s.t.} \;\;\;
f - \eta a + \mathcal{D}\varphi - h \in \cQ_\nu(\Gamma).
\end{equation}
The computational cost of the SOS programs~\cref{eq:vp-SOS,eq:L-sos-boundary-free} rises with increasing dimension $n$ of $\Omega$, dimension $m$ of $u$, maximum degree $\nu$, and, in the case of~\cref{eq:vp-SOS}, the number of boundary components $s$. When these integers are not too large, the value of $\Lsos$ or $\HLsos$ can be computed numerically to obtain an explicit lower bound on $\cF^*$. This generally is not possible for the other problems in the sequence $\cL^*\ge\Lpdr\ge\Lnu\ge\Lsos$ of dual relaxations.

Raising the degree $\nu$ enlarges the polynomial spaces in~\cref{eq:vp-finite-dimensional} and~\cref{eq:vp-SOS}, so
\begin{alignat}{4}
\Lpdr\ge\Linf&&\;\ge\cdots\ge\LnuOne \; && \; \geq \; && \Lnu \; \ge \cdots \phantom{,} \nonumber  \\
\label{eq:L-relaxations}\cwgeq&&\cwgeq \;\,&& && \cwgeq \quad\quad\quad\;\; \\ \nonumber
\Lsosinf&&\ge\cdots\ge\;\LsosOne && \;\geq \; && \Lsos \ge \cdots,
\end{alignat}
where the $\infty$ subscript denotes the $\nu\to\infty$ limit. (Analogous inequalities relate the boundary-free relaxations $\HLpdr$, $\HLnu$ and $\HLsos$.) In general $\Lsos$ and $\Lnu$ are strictly smaller than $\Lpdr$ for fixed $\nu$, but one hopes that $\Lsosinf=\Lpdr$. When this is true and $\cF^*=\Lpdr$, computing $\Lsos$ with increasing $\nu$ gives arbitrarily sharp lower bounds on the global minimum $\cF^*$ of the original variational problem~\cref{eq:primal-vp}. Two theorems that guarantee $\Lsosinf=\Lpdr$ under different conditions are stated in \cref{sec:sos-conv} below. First, \cref{sec:simple-variational-problem-sos} demonstrates convergence in a computational example.

\begin{remark}
\label{rem:rational}
The SOS program \cref{eq:vp-SOS} can be generalized for non-polynomial $\varphi$ that are semialgebraic functions. In the case of rational $\vphi$ where the numerator is a tunable polynomial and the denominator $D$ is a fixed nonnegative polynomial, one simply multiplies the nonnegativity constraints~\cref{eq:pw-inequality-boundary} and~\cref{eq:pw-inequality-bulk} by $D$ and $D^2$, respectively, to obtain polynomial inequalities that can be strengthened into SOS conditions. It can be shown that for any $\Lsos$ value obtained with polynomial $\vphi$, there exists a degree $\nu'$ (dependent on $D$) for the numerator such that rational $\vphi$ give a value at least as large as $\LsosWithArg{\nu}$. Therefore, convergence $\Lsos\nearrow\Lpdr$ established in \cref{sec:sos-conv} for polynomial $\vphi$ 
	implies convergence with rational $\vphi$ for any fixed $D$, as long as zeros of $D$ do not introduce singularities that make $\vphi$ inadmissible. Similar statements apply to the boundary-free formulation and to the SOS programs of \cref{ss:sos-int-ineq}. As demonstrated in \cref{sec:computational}, using rational $\vphi$ with a good choice of $D$ can greatly accelerate convergence of the various SOS relaxations.
\end{remark}

\subsubsection{\label{sec:simple-variational-problem-sos}A computational example}

Consider once again the variational problem \cref{eq: intro ex 1}, for which it was shown in \cref{ex:simple-variational-problem} that both $\HLpdr$ and $\Lpdr$ are sharp lower bounds on $\cF^*=2(\pi/3-\sqrt{3})$. The optimal $\vphi(x,y)$ giving $\HLpdr$ and the optimizing sequence of $\vphi(x,y)$ giving $\Lpdr$ in \cref{ex:simple-variational-problem} are polynomial in $y$ but not in $x$, so they cannot be found by solving the SOS programs~\cref{eq:vp-SOS} or~\cref{eq:L-sos-boundary-free}. To show that polynomial $\vphi$ found by computing $\Lsos$ nonetheless give excellent lower bounds on $\Lpdr$, we have computed $\Lsos$ for all odd polynomial degrees $\nu\le101$.

For this example, the semialgebraic sets~\cref{eq:Gamma-Lambda-semialg} must encode the domain $(-\frac\pi3,\frac\pi3)$ and its endpoints. This can be done by letting $g_{1,2}(x)=\frac\pi3 \pm x$, so
\begin{subequations}
\begin{align}
\Gamma &= \{(x,y,z) \in \R^3: \, \tfrac{\pi}{3}+x\geq 0,\,\tfrac{\pi}{3}-x\geq 0 \}, \\
\Lambda_{1,2} &= \{ (x,y) \in \R^2: \,\tfrac{\pi}{3}\pm x= 0,\,\tfrac{\pi}{3}\mp x\geq 0,\, y=0 \} = \{(\mp \tfrac{\pi}{3},0)\}.
\end{align}
\end{subequations}
Since $\Lambda_1$ and $\Lambda_2$ are singletons, the corresponding weighted SOS constraints in~\cref{eq:vp-SOS} can be replaced by inequalities in which the polynomials $q_1,q_2$ are simply numbers. The Lagrange multiplier $\eta$ in~\cref{eq:vp-SOS} does not enter because the variational problem~\cref{eq: intro ex 1} has no integral constraints. With these simplifications,~\cref{eq:vp-SOS} becomes
\begin{align}
\label{simple-example-Q}
\Lsos =& \sup\limits_{\substack{\vphi \in \Phi\cap\mathbb{R}[x,y]_\nu \\h\in\R[x]_\nu,~q_1,q_2 \in \R}}
\bigg\{\int_{-\pi/3}^{\pi/3} h(x)\,\dx + q_1 + q_2 \bigg\} \\ \nonumber
&\text{s.t.} \quad \begin{array}[t]{l}
z^2 - y^2 - 2y + \vphi_x(x,y) + \vphi_y(x,y) z - h(x) \in\mathcal{Q}_\nu\left(\Lambda\right),\\
\varphi(-\pi/3,0) - q_1 \geq 0, \\
-\varphi(\pi/3,0) - q_2 \geq 0.
\end{array}
\end{align}
Polynomials $\vphi\in\Phi$ can be at most quadratic in $y$ since the variational problem~\cref{eq: intro ex 1} is posed over $u \in W^{1,2}$, and the arguments introducing $\vphi$ in~\cref{sec:pointwise-dual} rely on $x \mapsto \mathcal{D}\vphi(x,u(x),\nabla u(x))$ being integrable. We therefore let $\vphi(x,y) = F(x)y^2$ with tunable polynomial $F(x) \in \R[x]_\nu$. {The resulting SOS program was solved using the software described at the start of \cref{sec:computational}.}

\begin{figure}[t]
\center
\includegraphics[width=0.99\textwidth]{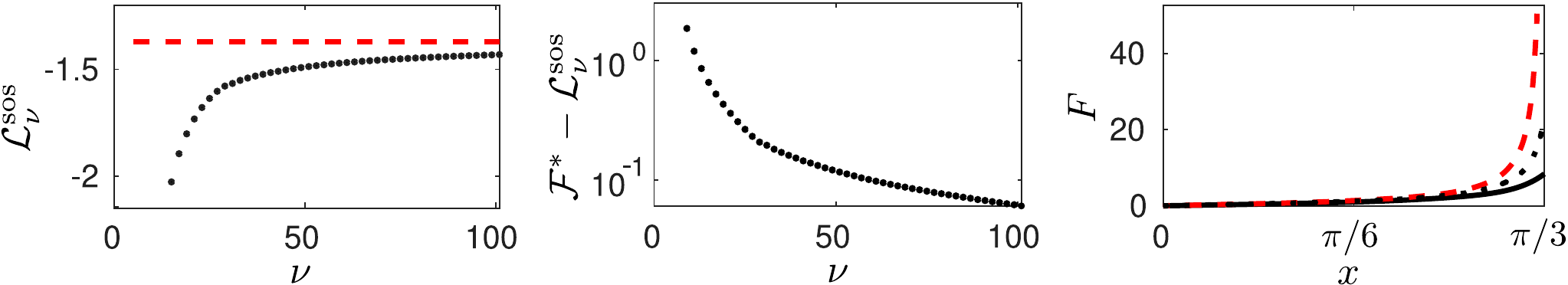}
\caption{Left: Lower bounds $\Lsos$ on the minimum $\mathcal{F}^*=2(\tfrac{\pi}{3}-\sqrt{3})$ (${\color{red}\dashedrule}$) of the variational problem~\cref{eq: intro ex 1}, computed by solving the SOS program~\cref{simple-example-Q} with $\vphi(x,y) = F(x)y^2$ and polynomial $F$ of odd degrees $\nu\le 101$. Middle: Gap between $\cF^*$ and $\Lsos$. Right: The optimal $F$ (${\color{red}\dashedrule}$) given by~\cref{eq:first-ex-phi}, along with $F$ that are optimal among polynomials of degree $\nu=11$ ($\solidrule$) and $\nu=21$ ($\dottedrule$). Each $F$ is odd and so is shown over only half its domain.}
\label{figure-simple-example}
\end{figure}

\Cref{figure-simple-example} shows the computed $\Lsos$ values for $\nu$ up to 101, as well as their decreasing distance from the value $\cF^*$ that they bound from below. These computations are carried out in (multiple precision) floating point arithmetic and so are subject to numerical error, but the SDP solver tolerances suggest errors smaller than plotting precision~\cite{SDPA}. Convergence of the exact $\Lsos$ values to $\cF^*$ is guaranteed by~\cref{th:sharpness-qm} below, although this theorem does not estimate the rate of convergence, which appears to be linear for large $\nu$.
As $\nu$ is raised, the $F(x)$ polynomials appear to converge pointwise to $F^*(x)=\frac{2 \sin x}{2\cos x - 1}$, which is the $x$-dependent part of the optimal $\vphi^*$ given in~\cref{eq:first-ex-phi}. Similarly, the $h$ polynomials (not plotted for brevity) appear to converge to the optimal $h^*$ in~\cref{eq:first-ex-phi}.

The relatively slow convergence of $\Lsos$ towards $\cF^*$ in this example is explained by the $F$ polynomials trying, in some sense, to approximate the function $F^*$ that is singular at $x=\pm\pi/3$. This naturally suggests that convergence of lower bounds may be accelerated by using \emph{rational} $\vphi$ that are singular at the boundaries; for the present example, we may take $\vphi(x,y) = F(x)y^2/(\pi^2/9-x^2)$. Such singular $\vphi$ can be used only in the boundary-free formulation $\HLpdr$. The SOS relaxation~\cref{eq:L-sos-boundary-free} can be generalized to this case as explained in \cref{rem:rational}, and it exhibits significantly faster convergence as $\nu$ is raised: solving the SOS program with rational $\vphi$ of numerator degree $\nu=5$ gives a lower bound on $\cF^*$ that is sharp to 4 digits, which is already much better than $\HLsos$ for polynomial $\vphi$ of degree $\nu=101$.

\subsubsection{\label{sec:sos-conv}Convergence of SOS relaxations}
Lower bounds $\Lsos$ on $\Lpdr$ obtained with the SOS relaxations described above can be guaranteed to converge to $\Lpdr$ in at least two cases.
The first case requires \emph{a priori} constraints that place the values of $(x,u,\nabla u)$ in a compact subset of $\overline\Omega\times\R^m\times\R^{m\times n}$. This is made precise by the following theorem, whose proof is given in \cref{ss:proof-sos-variational}. The proof follows a standard argument where near-optimal solutions to the PDR problem~\cref{eq:pdr-2} are approximated by polynomials, and then Putinar's Positivstellensatz~\cite{Putinar} guarantees weighted SOS representations of polynomials that are positive on compact sets.

\begin{theorem}
\label{th:sos-variational-problem}
If there exist $r_0,\ldots,r_s \in \R$ such that $r_0^2 - |(x,y,z)|^2 \in \mathcal{Q}(\Gamma)$ and $r_i^2 - |(x,y)|^2 \in \mathcal{Q}(\Lambda_i)$ for $i=1,\ldots,s$, then $\Lsos\nearrow\Lpdr$ as $\nu\to\infty$.
\end{theorem}

\begin{remark}
\label{rem:compactness}
The assumption of \cref{th:sos-variational-problem} implies $|(x,y,z)|^2\le r_0^2$ on $\Gamma$ and $|(x,y)|^2\le r_i^2$ on $\Lambda_i$, so a necessary (but not sufficient) condition for the theorem to apply is that the sets $\Gamma$ and $\Lambda_i$ are compact. This can happen only if the variational problem~\cref{eq:primal-vp} has pointwise constraints that uniformly bound $|u|$ and $|\nabla u|$. Constraints implying uniform bounds can often be added to a variational problem, but usually it is unclear \emph{a priori} how large these bounds should be in order to not change the optimizer. When pointwise constraints do render $\Gamma$ and $\Lambda_i$ compact, \cref{th:sos-variational-problem} either applies or can be made applicable. Although it is possible for $\Gamma$ and $\Lambda_i$ to be compact without having $r_0^2-|(x,y,z)|^2\in\mathcal{Q}(\Gamma)$ and $r_i^2-|(x,y)|^2\in\mathcal{Q}(\Lambda_i)$~\cite[Chapter~2]{Lasserre2015}, in such cases one can add the constraints $r_0^2-|(x,y,z)|^2\ge0$ and $r_i^2-|(x,y)|^2\ge0$ to the semialgebraic definitions~\cref{eq:Gamma-Lambda-semialg} of the sets $\Gamma$ and $\Lambda_i$, respectively, with $r_0$ and $r_i$ sufficiently large that the compact sets themselves do not change. This meets the assumptions of \cref{th:sos-variational-problem} without changing $\cF^*$ or $\Lpdr$; only $\Lsos$ changes by the enlargement of the polynomial sets $\cQ_\nu(\Gamma)$ and $\cQ_\nu(\Lambda_i)$ in~\cref{eq:vp-SOS}.
\end{remark}

The second case in which we can guarantee $\Lsos \nearrow \Lpdr$ is when $u$ satisfies homogenous Dirichlet boundary conditions and the PDR problem~\cref{eq:pdr-2} admits an optimizing sequence such that inequality~\cref{eq:pw-inequality-bulk} is polynomial in $(y,z)$ and is strictly satisfied in the sense of the second condition in \cref{th:sos-convergence-noncompact} below. 
The assumptions of this theorem hold, for instance, for the variational problem of \cref{ex:simple-variational-problem} with the optimizing sequence constructed there. 

\begin{theorem}
\label{th:sos-convergence-noncompact}
Suppose that $u$ vanishes on all boundary components $\partial\Omega_i$. Then, $\Lsos \nearrow \Lpdr$ as $\nu \to \infty$ if:
	\begin{enumerate}[leftmargin=*,labelwidth=1ex]
		\item There exists $r \in \R$ such that $r^2 - \|x\|^2 \in \mathcal{Q}(\Omega)$.
		\item Problem~\cref{eq:pdr-2} admits a maximizing sequence $\{(\eta^k,\varphi^k,h^k,\ell_1^k,\ldots,\ell_s^k)\}_{k \geq 1}$ where $\varphi^k(x,y)$ is polynomial in $y$ and where there exists a polynomial vector $m(y,z)$ and continuous symmetric-matrix-valued function $H_k:\overline\Omega\to\mathbb{S}^n$ with $H_k(x)\succ0$ on $\Omega$ such that, for all $(x,y,z)\in\Gamma$,
		\begin{equation}\label{e:matrix-representation-assumption}
		f(x,y,z) + \mathcal{D}\varphi^k(x,y,z) -\eta^k a(x,y,z)- h^k(x) = m(y,z)^\mathsf{T} H_k(x) m(y,z).
		\end{equation}
\end{enumerate}
\end{theorem}

\begin{proof}
Details of the proof are given in \cref{ss:proof-sos-convergence-noncompact}. Briefly, since $\Omega$ is compact and $\varphi^k$ is polynomial in $y$, the Weierstrass approximation theorem can be used to show that there exist polynomials $(\eta^k,\varphi^k,h^k,q_1^k,\ldots,q_s^k)$ such that:
\begin{enumerate}[(a), leftmargin=*,labelwidth=4pt]
\item\label{condition-opt} $\int_\Omega h^k \,\dx+\sum_{i=1}^s \int_{\partial\Omega_i} q_i^k \abs{\nabla g_i} \,\dbound \nearrow \Lpdr$,
\item\label{condition-bnd} $\varphi^k \cdot \nabla g_i - q_i^k\abs{\nabla g_i}^2 > 0$ on $\Lambda_i = \partial\Omega_i \times \{0\}$, and
\item\label{condition-bulk} Identity \cref{e:matrix-representation-assumption} holds with a positive definite polynomial matrix $H_k$.
\end{enumerate}
Condition~\ref{condition-opt} implies that $\Lsos \nearrow \Lpdr$ if $(\eta^k,\varphi^k,h^k,q_1^k,\ldots,q_s^k)$ is feasible for the SOS program~\cref{eq:vp-SOS} for some $\nu$. This is true because Putinar's Positivstellensatz~\cite{Putinar} and condition~\ref{condition-bnd} imply $\varphi^k \cdot \nabla g_i - q_i^k\abs{\nabla g_i}^2 \in \mathcal{Q}_{\nu}(\Lambda_i)$ for large enough $\nu$. Similarly, thanks to condition~\ref{condition-bulk}, one can apply to $H_k$ a Positivstellensatz for polynomial matrices that are positive definite on compact sets~\cite{Scherer2006}  to obtain $f + \mathcal{D}\varphi^k -\eta^k a- h^k \in \mathcal{Q}_{\nu}(\Gamma)$ for large enough $\nu$.
\end{proof}

To show convergence $\HLsos \nearrow \HLpdr$ for the boundary-free relaxations, \Cref{th:sos-variational-problem} cannot easily be adapted because its proof relies on uniform polynomial approximation of possibly non-polynomial $\vphi \in \Phi$. This is not generally possible when the set $\Phi$ is replaced by $\widehat{\Phi}$, which includes singular $\vphi$. With homogenous Dirichlet boundary conditions, however, the following theorem extends \cref{th:sos-convergence-noncompact} to the boundary-free case, provided that $\HLpdr$ admits an optimizing sequence of \emph{nonsingular} $\vphi$.

\begin{theorem}
\label{cor:L-sos-conv}
Suppose that $u$ vanishes on all boundary components $\partial\Omega_i$. Then, $\HLsos \nearrow \HLpdr$ as $\nu \to \infty$ if:
\begin{enumerate}[leftmargin=*,labelwidth=1ex]
	\item There exists $r \in \R$ such that $r^2 - \|x\|^2 \in \mathcal{Q}(\Omega)$.
	\item Assumption 2 of \cref{th:sos-convergence-noncompact} holds with each $\varphi^k \in \Phi \cap \widehat\Phi$.
\end{enumerate}
\end{theorem}
\begin{proof}
	Assumption 2 implies $\vphi^k(x,y) = \sum_{\alpha} c_\alpha(x) y^\alpha$ for some coefficients $c_\alpha \in C^1(\overline{\Omega}; \mathbb{R}^n)$ with $\sum_{i=1}^s \int_{\partial\Omega_i}c_0(x) \cdot n_i(x) \dbound=0$. The latter sum is a continuous linear functional on $C^1(\overline{\Omega}; \mathbb{R}^n)$, so an extension~\cite[Proposition 2]{Peet2007} of Weierstrass's approximation theorem lets the $c_\alpha$ be approximated by polynomials $p_\alpha$ satisfying $\sum_{i=1}^s \int_{\partial\Omega_i}p_0(x) \cdot n_i(x) \dbound=0$. This gives a polynomial $\vphi \in \widehat\Phi$ approximating $\vphi^k$, after which the proof can be completed using the argument in \cref{ss:proof-sos-convergence-noncompact}.
\end{proof}

\subsection{Optimization subject to parametrized integral inequalities}
\label{ss:sos-int-ineq}

Recall that the optimization problem~\cref{eq:integral-inequality-setup} defining $\cB^*$ can be relaxed into $\Bpdr$ or $\HBpdr$ by replacing the integral inequality $\cF^*(\lambda)\ge0$ with the stronger constraints $\Lpdr(\lambda)\ge0$ or $\HLpdr(\lambda)\ge0$. Strengthening these constraints further by requiring nonnegativity of the SOS relaxations $\Lsos(\lambda)$ or $\HLsos(\lambda)$ defined in \cref{eq:vp-SOS,eq:L-sos-boundary-free} gives
\begin{subequations}
	\begin{align}
	\label{eq:int-ineq-SOS}
	\Bpdr &\ge \Bsos :=
	\sup_{ \substack{\lambda\in\R^\tau,\,\eta \in \R_+\\\vphi \in \Phi \cap \R^n[x,y]_\nu\\h, q_i \in \R[x]_\nu} }
	b(\lambda) \quad \text{s.t.} \quad
	\begin{array}[t]{l}
	f + \mathcal{D}\varphi - \eta a - h \in \mathcal{Q}_\nu(\Gamma), \\
	\varphi\cdot \nabla g_i - q_i\abs{\nabla g_i}^2 \in \mathcal{Q}_\nu(\Lambda_i), \\ 
	L(h,q_1,\ldots,q_s) \geq 0,
	\end{array}
	\\[2ex]
	\label{eq:int-ineq-SOS-boundary-free}
	\HBpdr &\ge \HBsos :=
	\sup_{ \substack{\lambda \in \R^\tau,~\eta \in \R_+,\\ \vphi \in \widehat\Phi \cap \R^n[x,y]_\nu, \\ h \in \R[x]_\nu}} b(\lambda) \quad \text{s.t.} \quad
	\begin{array}[t]{l}
	f + \mathcal{D}\varphi - \eta a - h \in \mathcal{Q}_\nu(\Gamma), \\
	L(h) \geq 0.
	\end{array}
	\end{align}
\end{subequations}
Here, $L(h,q_1,\ldots,q_s)$ is a known linear function of the tunable coefficients of $h$ and the $q_i$, as described after~\cref{eq:L-objective}, and $L(h)$ simply omits the $q_i$. Problems \cref{eq:int-ineq-SOS,eq:int-ineq-SOS-boundary-free} are SOS programs if the objective $b(\lambda)$ is linear, and they can be converted to SOS programs if $b(\lambda)$ is a concave quadratic function. These cases, therefore, can be implemented using existing software for SOS programming. For more general concave $b(\lambda)$, problems \cref{eq:int-ineq-SOS,eq:int-ineq-SOS-boundary-free} can be converted into convex problems with linear inequalities and semidefinite constraints, which may be solved with standard interior-point algorithms if the gradient and Hessian of $b$ can be computed efficiently.

The assumptions needed to prove $\Lsos\nearrow\Lpdr$ in \cref{th:sos-variational-problem,th:sos-convergence-noncompact} can be strengthened slightly to conclude that $\Bsos\nearrow\Bpdr$  in at least two cases. The first case requires the assumptions of \cref{th:sos-variational-problem}, which cannot hold unless $(u,\nabla u)$ values are constrained to lie in compact sets, along with a particular ``strict feasibility'' condition. This is made precise by \cref{th:sos-integral-inequality} below. The second case requires the assumptions of \cref{th:sos-convergence-noncompact} along with an assumption about either strict feasibility or the form of the optimizer. This is made precise by \cref{th:sos-convergence-noncompact-int-ineq}.

\begin{theorem}
\label{th:sos-integral-inequality}
Suppose the assumptions of~\cref{th:sos-variational-problem} hold.
Then, $\Bsos\nearrow\Bpdr$ as $\nu\to\infty$
if there exist $(\lambda, \eta, \vphi, h, \ell_1, \ldots, \ell_s)$ feasible for~\cref{eq:Bpdr} and $\gamma>0$ such that
\begin{equation}
\label{eq:strict-feasibility-condition}
\int_{\Omega} h(x) \, \dx + \sum_{i=1}^s \int_{\partial\Omega_i} \ell_i \,\dbound \geq \gamma.
\end{equation}
\end{theorem}

\begin{theorem}
\label{th:sos-convergence-noncompact-int-ineq}
Suppose that $u$ vanishes on all boundary components $\partial\Omega_i$. Then, $\Bsos\nearrow\Bpdr$ as $\nu \to \infty$ if the assumptions of \cref{th:sos-convergence-noncompact} are satisfied and at least one of the following two conditions holds:
\begin{enumerate}[leftmargin=5ex,labelwidth=1ex]
\item[\rm 3a.] The maximizing sequence satisfying assumption {\rm 2} of \cref{th:sos-convergence-noncompact} has all $\ell_i^k = 0$ and has $\vphi^k(x,y)$ depending polynomially on $y$ with $\vphi^k(x,0)=0$.
\item[\rm 3b.] There exist $(\lambda^0,\eta^0,\varphi^0,h^0,\ell_{1}^0,\ldots,\ell_{s}^0)$ feasible for~\cref{eq:Bpdr} and $\gamma > 0$ such that $\varphi^0(x,y)$ depends polynomially on $y$ and
\beq \int_\Omega h^0(x) \,\dx +\sum_{i=1}^s \int_{\partial\Omega_i} \ell_i^0(x) \,\dbound \geq \gamma.
\eeq
\end{enumerate}
\end{theorem}

The proof of \cref{th:sos-integral-inequality} is given in~\cref{ss:proof-sos-integral-inequality}. It is a standard argument, similar to the proof of \cref{th:sos-variational-problem}, where one takes convex combinations of near-feasible points with the strictly feasible point satisfying~\cref{eq:strict-feasibility-condition}, then approximates these combinations with polynomials and applies Putinar's Positivstellensatz~\cite{Putinar}. The proof of \cref{th:sos-convergence-noncompact-int-ineq} is analogous but uses the approximation arguments and Positivstellens\"atze from the proof of \cref{th:sos-convergence-noncompact} rather than \cref{th:sos-variational-problem}. We omit the proof of \cref{th:sos-convergence-noncompact-int-ineq} for brevity but note its one complication beyond the proof of \cref{th:sos-convergence-noncompact}: under assumption 3a, one must approximate non-polynomial $h^k$ with polynomial $h$ satisfying $\int_{\Omega} h(x) \dx = \int_{\Omega} h^k(x) \dx$ by invoking the same extension of Weierstrass's approximation theorem used to prove \cref{cor:L-sos-conv}. This same approximation argument can be combined with \cref{cor:L-sos-conv} to obtain a convergence guarantee $\HBsos\nearrow\HBpdr$ for the boundary-free relaxations, as stated in the following corollary. We observe convergence of $\HBsos$ to $\HBpdr$ (in fact, to $\cB^*$) in all computational examples of \cref{sec:computational}, even when the assumptions of the corollary are not satisfied.

\begin{corollary}
\label{cor:B-sos-conv}
Under the assumptions of \cref{cor:L-sos-conv}, $\HBsos\nearrow\HBpdr$.
\end{corollary}

\section{\label{sec: sharp}Sharpness of relaxations for three classes of problems}

This section proves that the PDR method is sharp for three classes of problems: integral variational problems with quadratic integrand, optimization problems giving principal eigenvalues of elliptic operators, and optimization problems giving optimal constants of the Poincar\'e inequality for $L^p$ norms on one-dimensional domains with even $p$. In all cases, admissible functions $u$ are constrained to vanish at the boundaries. For each class of problem, SOS relaxations also give arbitrarily sharp bounds as their polynomial degrees are raised. These sharpness results rely on optimal $\vphi$ constructed semi-explicitly in terms of solutions to Euler--Lagrange equations.

\subsection{\label{sec:quadratic-minimization}Variational problems with quadratic integrand} Consider integral variational problems of the form
\begin{equation}\label{e:qm}
\mathcal{F}^* = \inf_{u \in W_0^{1,2}(\Omega)} \int_{\Omega} \left[ \nabla u \cdot A(x) \nabla u + b(x) u^2 - 2 c(x) u \right] \dx,
\end{equation}
where $W_0^{1,2}$ is the subspace of $W^{1,2}$ containing functions $u$ that vanish at the boundary of $\Omega$. We make the following assumptions:
\begin{enumerate}[{({A}1)}, leftmargin=*, topsep=0.75ex, itemsep=0.1ex]
	\item\label{ass:qm1} $\Omega$ is an open bounded Lipzchitz domain.
	\item\label{ass:qm2} $A:\R^n \to \mathbb{S}^n$ is smooth, where $\mathbb{S}^n$ is the space of $n\times n$ symmetric matrices.
	\item\label{ass:qm3} Uniform ellipticity: there exists $\omega>0$ such that $\omega I \preceq A(x) \preceq \omega^{-1}I$  $\forall x \in \overline{\Omega}$.
	\item\label{ass:qm4} $b:\R^n \to \mathbb{R}$ is smooth.
	\item\label{ass:qm5} $c:\R^n \to \mathbb{R}$ is smooth and strictly positive on $\overline{\Omega}$.
	\item\label{ass:qm6} $A$ and $b$ are such that
\end{enumerate}
\begin{equation}\label{e:qm:eig}
\lambda^* = \min_{\substack{u \in W_0^{1,2}(\Omega) \\ \int_\Omega u^2 \dx = 1}} \int_{\Omega} \left[\nabla u \cdot A(x) \nabla u + b(x) u^2 \right] \dx > 0.
\end{equation}
The last assumption \ref{ass:qm6} can be verified by using the PDR approach to bound $\lambda^*$ below; these bounds can be made arbitrarily sharp, as proved in the next subsection.

Under assumptions \ref{ass:qm1}--\ref{ass:qm6}, the global minimum $\mathcal{F}^*$ of \cref{e:qm} is attained by a unique minimizer $u^*$ \cite[Chapter~8]{Evans1998} that solves the elliptic Euler--Lagrange PDE
\begin{equation}\label{e:qm:EL-Dirichlet-problem}
-\nabla \cdot (A(x) \nabla u^*) + b(x) u^* = c(x) \text{ on }\Omega,
\end{equation}
subject to the Dirichlet boundary condition $u^*= 0$ on $\partial\Omega$. Provided the domain $\Omega$ is not too complicated or high-dimensional, this PDE can be solved by standard numerical methods, and then $\cF^*$ can be calculated as $- \int_{\Omega} c(x) u^*(x) \, \dx$. Alternatively, the following \cref{th:sharpness-qm} guarantees that PDRs bound $\cF^*$ below with arbitrary accuracy. The proof of this theorem exploits properties of the optimizer $u^*$ for problems of the form~\cref{e:qm} on Lipschitz domains, which are summarized in \cref{prop:qm:positive-minimizer}.

\begin{theorem}\label{th:sharpness-qm}
Suppose problem~\cref{e:qm} satisfies assumptions \ref{ass:qm1}--\ref{ass:qm6}. Then:
\begin{enumerate}[{\rm 1.}, leftmargin=*, topsep=0.5ex, itemsep=0.5ex]
\item $\Lpdr = \HLpdr = \cF^*$. In particular, there exist $h^k \in C(\overline{\Omega})$ and $F^k \in C(\overline{\Omega}; \mathbb{R}^n)$ such that, with $\vphi^k =F^k(x) y^2$, the sequence $\{(\vphi^k,h^k)\}_{k \geq 1}$ is maximizing for $\HLpdr$ and $\{(\vphi^k,h^k,0,\ldots,0)\}_{k \geq 1}$ is maximizing for $\Lpdr$.
\item If $A,b,c$ are polynomial and $\Omega$ is semialgebraic with $r^2 - \|x\|^2 \in \mathcal{Q}(\Omega)$ for some $r\in\R$, then, with $\vphi(x,y) = F(x)y^2$ and degree-$\nu$ polynomials $F(x)$ and $h(x)$, the SOS relaxations $\Lsos$ and $\HLsos$ converge to $\cF^*$ as $\nu\to\infty$.
\end{enumerate}
\end{theorem}

\begin{proposition}\label{prop:qm:positive-minimizer}
Let $\Omega$ be an open bounded Lipschitz domain. The optimizer $u^*$ for problem~\cref{e:qm} is smooth and strictly positive on $\Omega$.
\end{proposition}
\begin{proof}[Proof of \cref{prop:qm:positive-minimizer}]
The smoothness of $u^*$ is given by regularity theory for elliptic PDEs~\cite[\S6.3.2]{Evans1998}. For positivity of $u^*$, note first that optimal $\lambda$ and $u$ for \cref{e:qm:eig} depend continuously on the domain, as can be shown using the arguments in~\cite{Savare2002}. Fix any smooth domain $\Omega'\supset\Omega$ such that the minimum $\lambda'$ of problem~\cref{e:qm:eig} on $\Omega'$ is strictly positive. The corresponding minimizer $w$ is smooth, positive on $\Omega'$, and satisfies $-\nabla \cdot (A \nabla w) + b w = \lambda' w$~\cite[\S6.5.1]{Evans1998}.
	This equation for $w$ and equation~\cref{e:qm:EL-Dirichlet-problem} for $u^*$ imply that the smooth function $v=u^*/w$ satisfies the elliptic PDE $-\nabla \cdot (A \nabla v) -2 w^{-1} \nabla w \cdot A \nabla v + \lambda' w = w^{-1}c$ on $\Omega$.
	Since $\lambda'>0$, the strong maximum principle~\cite[\S6.4.2]{Evans1998} implies $v>0$ on $\Omega$, and thus $u^* = wv>0$ on~$\Omega$.
\end{proof}

\begin{proof}[Proof of \cref{th:sharpness-qm}]
We construct a sequence $\{(\varphi^k,h^k)\}_{k\ge 1}$ that satisfies the constraint of~\cref{eq:pdr-no-boundary} and has $\int_{\Omega}h^k(x) \dx \nearrow \cF^*$, which shows $\HLpdr = \cF^*$. In the case of polynomial data, for which $\HLsos$ is defined, our sequence also satisfies the assumptions of \cref{cor:L-sos-conv}, so that $\HLsos \nearrow \HLpdr$. The corresponding sequence $\{(\vphi^k,h^k,0,\ldots,0)\}_{k \geq 1}$ satisfies the constraints of~\cref{eq:pdr-2}, so $\Lpdr = \cF^*$, and for polynomial data it also satisfies the assumptions of \cref{th:sos-convergence-noncompact}, so $\Lsos \nearrow \Lpdr$.

	With $\vphi^k(x,y)=F^k(x)y^2$, the constraint of~\cref{eq:pdr-no-boundary} for problem~\cref{e:qm} is inequality~\cref{eq:pw-inequality-bulk}:
	\begin{equation}\label{e:qm:inequality}
	z^\mathsf{T}A(x) z + b(x)y^2 - 2c(x) y + \nabla \cdot F^k(x) y^2 + 2y F^k(x)\cdot z - h^k(x) \ge 0
	\end{equation}
	for all $(x,y,z)\in\Gamma=\Omega\times\mathbb{R}\times\mathbb{R}^n$. (The multiplier $\eta$ in \cref{eq:pw-inequality-bulk} does not enter since there is no integral constraint in~\cref{e:qm}.) This inequality can be written in the matrix form
	$m(y,z)^\top H_k(x) m(y,z) \geq 0$ with $m(y,z)=\begin{bmatrix} 1 & y & z\end{bmatrix}^\top$ and
	\begin{equation}
	H_k(x) :=
	\begin{bmatrix}
	-h^k(x) & -c(x) & 0_{1\times n} \\
	-c(x) & b(x) + \nabla \cdot F^k(x)  & F^k(x)^\top \\
	0_{n\times 1} & F^k(x) & A(x)
	\end{bmatrix}.
	\end{equation}
	To satisfy both \cref{e:qm:inequality} and assumption 2 of \cref{cor:L-sos-conv}, we seek functions $F^k$ and $h^k$ such that $H_k(x)$ is positive definite on $\overline{\Omega}$. Since $A(x)$ is positive definite on $\overline{\Omega}$ by~\ref{ass:qm3}, it suffices to check the positive definiteness of the Schur complement
	\begin{equation}\label{e:HkAmatrix}
	H_k(x)/A(x) =
	\begin{bmatrix}
	-h^k(x) & -c(x) \\
	-c(x) & b(x) + \nabla \cdot F^k(x) - F^k(x) \cdot {A^{-1}}(x) F^k(x)
	\end{bmatrix}.
	\end{equation}

Let $\Omega_k\supset\Omega$ denote the $k^{-1}$ neighborhood of $\Omega$, meaning the set of points whose distance from $\Omega$ is less than $k^{-1}$. The set $\Omega_k$ is open and Lipschitz for sufficiently large $k$ \cite{Doktor1976}. We choose
\begin{equation}\label{e:Fkformula}
F^k(x) = -\frac{A(x) \nabla u^k(x)}{u^k(x)},
\end{equation}
where $u^k$ is the minimizer of problem~\cref{e:qm} on domain $\Omega_k$. Let $k$ be large enough for $\Omega_k$ to be Lipschitz and for assumptions \ref{ass:qm1}--\ref{ass:qm6} to hold with $\Omega$ replaced by $\Omega_k$. The function $u^k$ satisfies the PDE~\cref{e:qm:EL-Dirichlet-problem} on $\Omega_k$ and is both smooth and positive on $\overline{\Omega} \subset \Omega_k$ by \cref{prop:qm:positive-minimizer}.
	With $F^k$ as in~\cref{e:Fkformula} and $u^k$ satisfying~\cref{e:qm:EL-Dirichlet-problem}, the lower right entry of~\cref{e:HkAmatrix} becomes $c(x)/u^k(x)$, which is a strictly positive continuous function on $\overline{\Omega}$ by virtue of assumption \ref{ass:qm5} and the strict positivity of $u_k$. We choose $h^k(x) = - c(x) u^k(x) - k^{-1}$, in which case the matrix $H_k/A$ is positive definite on $\overline{\Omega}$ because its determinant and trace are positive.
	
To show $\int_{\Omega} h^k(x) \dx \nearrow \cF^*$, observe that
\begin{equation}
\int_{\Omega} h^k(x) \, \dx
 = \int_{\Omega} \left[- c(x) u^k(x) - k^{-1}\right] \dx
 \geq  \int_{\Omega_k} \left[- c(x) u^k(x) - k^{-1}\right] \dx
\end{equation}
since $cu^k\geq 0$ on $\Omega_k$. The right-hand integral is the minimum of~\cref{e:qm} on the domain $\Omega_k$, by the definition of $u^k$. This integral converges to $\cF^*$ as $k \to \infty$ because minimizers of~\cref{e:qm} depend continuously on the domain, as follows since minimizers solve elliptic Euler--Lagrange equations with continuous domain dependence \cite{Savare2002}.
\end{proof}

\subsection{\label{ss:elliptic-eigenvalues}Principal eigenvalues of elliptic operators}
Consider the optimization
\begin{equation}\label{e:princ-eig}
\lambda^* := \sup_{\lambda \in \mathbb{R}} \lambda \quad\text{s.t.}\quad \inf_{u \in W^{1,2}_0}\int_{\Omega} \left(\nabla u \cdot A(x) \nabla u + [b(x)- \lambda c(x)] u^2 \right) \dx \geq 0,
\end{equation}
where a parametrized variational integral appears in the constraint. Let $\Omega$, $A$, $b$, and $c$ satisfy the same assumptions \ref{ass:qm1}--\ref{ass:qm5} made in \cref{sec:quadratic-minimization}. The optimal value $\lambda^*$ is the principal eigenvalue of the elliptic eigenvalue problem
\begin{equation}\label{e:elliptic-eig-problem}
-\nabla \cdot (A(x) \nabla u) + b(x) u = \lambda\, c(x) u \quad \text{ on } \Omega,
\end{equation}
subject to the Dirichlet boundary condition $u=0$ on $\partial\Omega$. When $\Omega$ is a simple low-dimensional domain, $\lambda^*$ can be approximated numerically using standard discretization techniques for PDEs and eigenvalue computations. The following theorem guarantees that PDRs approximate $\lambda^*$ from below with arbitrary accuracy.

\begin{theorem}\label{th:elliptic-eigs}
	Suppose problem \cref{e:princ-eig} satisfies assumptions \ref{ass:qm1}--\ref{ass:qm5} in \cref{sec:quadratic-minimization}. Then:
\begin{enumerate}[{\rm 1.}, leftmargin=*, topsep=0.5ex, itemsep=0.5ex]
\item $\Bpdr = \HBpdr = \lambda^*$. In particular, there exist $\lambda^k \in \mathbb{R}$ and $F^k \in C(\overline{\Omega}; \mathbb{R}^n)$ such that, with $\vphi^k =F^k(x) y^2$, the sequence $\{(\lambda^k,\vphi^k,0)\}_{k \geq 1}$ is maximizing for $\HBpdr$ and $\{(\lambda^k,\vphi^k,0,0,\ldots,0)\}_{k \geq 1}$ is maximizing for $\Bpdr$.
\item If $A,b,c$ are polynomial and $\Omega$ is semialgebraic with $r^2 - \|x\|^2 \in \mathcal{Q}(\Omega)$ for some $r\in\R$, then, with $h=0$ and $\vphi(x,y) = F(x)y^2$ for degree-$\nu$ polynomial $F(x)$, the SOS relaxations $\Bsos$ and $\HBsos$ converge to $\lambda^*$ as $\nu\to\infty$.
\end{enumerate}
\end{theorem}

\begin{proof}
	The proof is similar to that of \cref{th:sharpness-qm}, so we only sketch the main ideas required to prove that $\HBpdr=\lambda^*$. Choose $h^k=0$, so the $\int_\Omega h^k(x) \dx \geq 0$ constraint in \cref{eq:Bpdr-hat} is satisfied. It remains to check inequality~\cref{eq:pw-inequality-bulk}, which for problem~\cref{e:princ-eig} and $\vphi(x,y) = F^k(x)y^2$ can be written in the matrix form
	\begin{equation}
	\begin{bmatrix}y \\ z\end{bmatrix}^\top
	\begin{bmatrix}
	-\lambda^k + \nabla \cdot F^k(x) & F^k(x)^\top \\
	F^k(x) & A(x)
	\end{bmatrix}
	\begin{bmatrix}y \\ z\end{bmatrix} \geq 0 \quad \forall(x,y,z) \in \Omega \times \mathbb{R} \times \mathbb{R}^n.
	\end{equation}
	Using the Schur complement as in the proof of \cref{th:sharpness-qm}, one can verify that the $x$-dependent matrix in this inequality is positive definite on $\overline{\Omega}$ after setting
	\begin{equation}
	F^k(x) = -\frac{A(x) \nabla u^k(x)}{u^k(x)} \quad \text{and}\quad \lambda^k = \mu^k - \frac1k,
	\end{equation}
	where $u^k$ and $\mu^k$ are the principal eigenfunction and eigenvalue for problem~\cref{e:elliptic-eig-problem} on $\Omega_k$, the $k^{-1}$ neighborhood of $\Omega$. Note that $F^k(x)$ is continuous on $\overline{\Omega}$ because $u^k$ is strictly positive on that set~\cite[\S6.5.1]{Evans1998}. The convergence $\lambda^k \nearrow \lambda^*$ as $k \to \infty$ follows from the continuity of principal eigenvalues of elliptic operators with respect to domain perturbations~\cite{Nirenberg,fleckinger1986,Henrot1994}, which ensures that $\mu^k \nearrow \lambda^*$.
\end{proof}

\subsection{Optimal  $L^{2q}$ Poincar\'{e} inequalities on intervals}
Consider the problem
\begin{equation}
\label{L2n-problem}
\lambda^* = \sup \lambda \;\text{ s.t. }
\displaystyle \inf_{u\in W^{1,2q}_0(\Omega)}
\int_\Omega \left( \abs{\nabla u}^{2q} - \lambda u^{2q}\right) \dx \ge 0,
\end{equation}
where $2q$ is a positive even integer. This $\lambda^*$ is the optimal constant in the Poincar\'e inequality $\Vert u \Vert_{L^{2q}}^{2q}\le\lambda^{-1}\Vert \nabla u\Vert_{L^{2q}}^{2q}$ for $u\in W_0^{2q}$,
as well as the principal eigenvalue of the nonlinear eigenproblem
\begin{equation}
-(2q-1) \nabla \cdot(|\nabla u|^{2q-2}\nabla u) = \lambda u^{2q-1}
\label{p-Laplacian}
\end{equation}
for the $2q$-Laplacian operator $\nabla \cdot(|\nabla u|^{2q-2}\nabla u)$~\cite{Lindqvist2008}. When the domain is a 1D interval, the following theorem guarantees that PDRs give arbitrarily sharp lower bounds on $\lambda^*$. This is of interest as it pertains to the sharpness of the PDR approach, although to compute $\lambda^*$ numerically there is a more explicit formula for 1D intervals \cite{Lindquist}.

\begin{theorem}
\label{thm:L2n}
For problem~\cref{L2n-problem} on a bounded open interval $\Omega\subset\R$:
\begin{enumerate}[{\rm 1.}, leftmargin=*, topsep=0.5ex, itemsep=0.5ex]
\item $\HBpdr=\Bpdr = \lambda^*$. In particular, there exist $F^k \in C(\overline{\Omega})$ and $\lambda^k \in \mathbb{R}$ such that with $\vphi^k =F^k(x) y^{2q}$ the sequence $\{(\lambda^k, \vphi^k,0)\}_{k \geq 1}$ is maximizing for $\HBpdr$, and $\{(\vphi^k,h^k,0,\ldots,0)\}_{k \geq 1}$ is maximizing for $\Bpdr$.
\item Consider a semialgebraic definition of $\Omega$ for which $r^2-x^2\in\mathcal{Q}(\Omega)$.
With $h=0$ and $\vphi(x,y) = F(x)y^{2q}$ for degree-$\nu$ polynomial $F(x)$, the SOS relaxations $\Bsos$ and $\HBsos$ converge to $\lambda^*$ as $\nu\to\infty$.
\end{enumerate}
\end{theorem}

\begin{proof}[Proof of~\Cref{thm:L2n}]
Set $h^k=0$, so the constraint $\int_{\Omega}h^k(x) \,\dx \geq 0$ in \cref{eq:Bpdr-hat} is satisfied. For any fixed integer $q \geq 1$, let $\vphi^k(x,y) = F^k(x)y^{2q}$. Inequality~\cref{eq:pw-inequality-bulk} reads 
\begin{equation}
z^{2q} - \lambda^k y^{2q} + F_x^ky^{2q} + 2qF^k(x)y^{2q - 1}z \ge 0
\label{L2pineq}
\end{equation}
and must hold for all $(x,y,z)\in \Gamma = \Omega\times\mathbb{R}\times\mathbb{R}$. For each $j = 1,2,\dots,q-1$, we will add and subtract the terms $2G^k_j(x)y^{2q - 2j}z^{2j}$ on the left-hand side of the pointwise inequality \cref{L2pineq}, so that we can rewrite it in matrix form $m(y,z)^\top H_k(x)m(y,z) \geq 0$ with $m(y,z) = \begin{bmatrix} y^q & y^{q-1}z & y^{q-2}z^2 & \dots & yz^{q-1} & z^q\end{bmatrix}^\top$ and
\begin{equation}
\arraycolsep=.5pt\def\arraystretch{1}
	\begin{split}
	&H^k(x)\!=\!\! \small\begin{bmatrix}
		F^k_x(x) - \lambda^k & qF^k(x) & -G^k_1(x) & 0 & \cdots & 0 & 0 & 0 \\
		qF^k(x) & 2G^k_1(x) & 0 & -G^k_2(x) & \cdots & 0 & 0 & 0 \\
		-G^k_1(x) & 0 & 2G^k_2(x) & 0 & \cdots & 0 & 0 & 0 \\
		0 & -G^k_2(x) & 0 & 2G^k_3(x) & \cdots & 0 & 0 & 0 \\
		\vdots & \vdots & \vdots & \vdots & \ddots & \vdots & \vdots & \vdots \\
		0 & 0 & 0 & 0 & \cdots & 2G^k_{q-2}(x) & 0 & -G^k_{q-1}(x) \\
		0 & 0 & 0 & 0 & \cdots & 0 & 2G^k_{q-1}(x) & 0 \\
		0 & 0 & 0 & 0 & \cdots & -G^k_{q-1}(x) & 0 & 1 \\
	\end{bmatrix}.
	\end{split}
\end{equation}
We seek functions $F^k,G^k_1,\dots,G^k_{q-1}$ and scalars $\lambda^k \nearrow \lambda^*$ such that $H_k(x)\succ0$ on $\overline\Omega$. While $H_k\succeq0$ suffices for part 1 of the theorem, $H_k\succ0$ lets~\cref{cor:B-sos-conv} be applied to ensure the convergence of SOS relaxations in part 2.

For each $k$ let
\begin{equation}
	F^k(x) = -(1 - 2\varepsilon)\bigg(\frac{u_x^k}{u^k}\bigg)^{2q-1} - \varepsilon, \quad G^k_j(x) = (1 - 2\varepsilon)(q - j)\bigg(\frac{u_x^k}{u^k}\bigg)^{2q - 2j} + \varepsilon
\end{equation}
with $\varepsilon>0$ to be chosen later. Let $\lambda^k = \mu^k - k^{-1}$, where $(\mu^k,u^k)$ is the principal eigenvalue--eigenfunction pair for the following nonlinear eigenproblem on the $k^{-1}$ neighborhood $\Omega_k\supset\Omega$,
\begin{equation}\label{pLapEqn}
	(1 - 2\varepsilon)(2q - 1)(u^k_x)^{2q - 2}u^k_{xx} + \mu^k(u^k)^{2q - 1} = 0\ \mathrm{on}\ \Omega_k, \quad u^k = 0\ \mathrm{on}\ \partial\Omega_k.
\end{equation}
Since $u^k$ is smooth and positive on $\Omega_k$~\cite{anane1987simplicite}, the functions $F^k$ and $G^k_j$ are continuous on $\overline{\Omega}$. Moreover, applying \cref{pLapEqn} gives $F^k_x(x) = \mu^k + (1-2\varepsilon)(2q - 1)(u_x^k/u^k)^{2q}$. The matrix $H^k(x)$ then can be written as 
$H^k(x) = H^k_1 + (1 - 2\varepsilon)D^k(x)H_2D^k(x)$,
where $D^k(x) = \mathrm{Diag}\big(\big[u^k_x/u^k\big]^q,\,\big[u^k_x/u^k\big]^{q-1},\,\ldots,\,1\big)$ and
\begin{subequations}
	\begin{equation}
	H^k_1= {\footnotesize \begin{bmatrix}
		\frac{1}{k}  & -q\varepsilon & -\varepsilon & 0 & 0 & \cdots & 0 & 0 & 0 \\
		-q\varepsilon & 2\varepsilon & 0 & -\varepsilon & 0 & \cdots & 0 & 0 & 0 \\
		-\varepsilon & 0 & 2\varepsilon & 0 & -\varepsilon  & \cdots & 0 & 0 & 0 \\
		0 & -\varepsilon & 0 & 2\varepsilon & 0 & \cdots & 0 & 0 & 0  \\
		0 & 0 & -\varepsilon & 0 & 2\varepsilon & \cdots & 0 & 0 & 0  \\
		\vdots & \vdots & \vdots & \vdots & \vdots & \ddots & \vdots & \vdots & \vdots \\
		0 & 0 & 0 & 0 & 0 & \cdots & 2\varepsilon & 0 & -\varepsilon \\
		0 & 0 & 0 & 0 & 0 & \cdots & 0 & 2\varepsilon & 0 \\
		0 & 0 & 0 & 0 & 0 & \cdots & -\varepsilon & 0 & 2\varepsilon \\
		\end{bmatrix} },
	\end{equation}
	\begin{equation}
	\begin{split}
	&H^k_2 = \footnotesize\begin{bmatrix}
	(2q -1)  & -q & -(q - 1) & 0 & 0 & \cdots & 0 & 0 & 0 \\
	-q & (2q - 2) & 0 & -(q - 2) & 0 & \cdots & 0 & 0 & 0 \\
	-(q - 1) & 0 & (2q - 4) & 0 & -(q-3)  & \cdots & 0 & 0 & 0 \\
	0 & -(q - 2) & 0 & (2q - 6) & 0 & \cdots & 0 & 0 & 0  \\
	0 & 0 & -(q-3) & 0 & (2q - 8) & \cdots & 0 & 0 & 0  \\
	\vdots & \vdots & \vdots & \vdots & \vdots & \ddots & \vdots & \vdots & \vdots \\
	0 & 0 & 0 & 0 & 0 & \cdots & 4 & 0 & -1 \\
	0 & 0 & 0 & 0 & 0 & \cdots & 0 & 2 & 0 \\
	0 & 0 & 0 & 0 & 0 & \cdots & -1 & 0 & 1 \\
	\end{bmatrix}.
	\end{split}
	\end{equation}
\end{subequations}
Choose any positive $\varepsilon < \min\{\frac{1}{2},k^{-1}(q+1)^{-1}\}$, so the matrix $H^k_1$ is diagonally dominant. Since its first and last rows are strictly diagonally dominant and $H^k_1$ is irreducible, we have $H^k_1\succ0$. Similarly, $H^k_2\succeq0$ because it has positive diagonal entries and is row diagonally dominant (all row sums vanish). The matrix $D^k(x)H_2^kD^k(x)$ is therefore positive semidefinite for all $x\in\overline\Omega$, so $H^k(x)\succ0$ on $\overline\Omega$. To conclude the proof, observe that $\mu^k\nearrow\lambda^*$ as $k\to \infty$ by the domain monotonicity and continuity properties of $2q$-Laplacian operators~\cite[Lemma 2.2]{diBlasio2020}, so $\lambda^k\nearrow\lambda^*$ also.
\end{proof}

\section{\label{sec:computational}Computational examples displaying convergence}

This section reports SOS computations giving lower bounds on the exact optima of various example problems. In each case, the computed bounds appear to become sharp as the polynomial degree is raised. This convergence is guaranteed by \cref{th:elliptic-eigs,thm:L2n} for the examples of \cref{sec: parabolic-example,sec: L4-example}, respectively, whereas in~\cref{sec:nonconvex} convergence is observed numerically but has not been proved by our analysis.
In all computations we use a modified version\footnote{Available from \url{https://github.com/aeroimperial-optimization/aeroimperial-yalmip}} of YALMIP~\cite{Lofberg2004,Lofberg2009} to reformulate SOS programs as SDPs and to interface with SDPA-GMP~\cite{SDPA}, which solves the SDP in multiple precision arithmetic to overcome numerical ill-conditioning.

\subsection{Optimal Poincar\'{e} inequalities for the $L^2$ norm on 2D domains}
\label{sec: parabolic-example}

For a first computational test of PDRs, consider problem~\cref{e:princ-eig} with
$A = I_{2\times2}$, $b=0$, and $c=1$.
We compute the SOS relaxations $\HBsos$ that bound $\lambda^*$ below. For $\Omega$ we choose two different 2D domains with corners, one a lens shape bounded by parabolas, and one a crescent bounded by circular arcs. These domains are shown in \cref{table:par-cs}, and their semialgebraic definitions are, respectively,
\begin{subequations}
\begin{align}
\label{parabolic-domain}
&\{ x\in\mathbb{R}^2 \colon\;  g_1(x) := 2x_1-x_2^2+1\ge 0, \;g_2(x) := 1-x_2^2-2x_1\ge 0\}, \\
\label{cs-domain}
&\{ x\in\mathbb{R}^2 \colon\;
	g_1(x) := 1-x_1^2-x_2^2 \ge 0,\; g_2(x) := (x_1+1)^2 + x_2^2 - 2 \ge 0 \}.
\end{align}
\end{subequations}
For each domain, our SOS computations seek rational $\vphi$ of the form
\beq
\label{eq:phi-poincare-2d}
\varphi(x,y) = \frac{P(x) y^2}{g_1(x)g_2(x)},
\eeq
where $P(x)$ is a degree-$\nu$ polynomial in $x\in\R^2$ whose coefficients are tunable in the SOS program. The derivation of polynomial SOS conditions for rational $\vphi$ is as described in \cref{rem:rational}. The rational form~\cref{eq:phi-poincare-2d} of $\vphi$ is motivated by the optimal $\vphi$ constructed in the proof of \cref{th:elliptic-eigs}, which suggests $\vphi(x,y)=F(x)y^2$ in the present example with $F(x)$ having a simple pole at the domain boundary.

The table in~\cref{table:par-cs} reports $\HBsos$ using these rational $\vphi$ for various $\nu$ in both domains. The values converge quickly to $\lambda^*$, which to the best of our knowledge is known analytically only for the lens-shaped domain~\cite{MorseFeshbach}. Since rational $\vphi$ generalize polynomial ones, the convergence $\HBsos\nearrow\lambda^*$ as $\nu\to\infty$ can be guaranteed as a corollary of \cref{th:elliptic-eigs}. The convergence rate, however, is much faster when $\vphi$ is rational as in~\cref{eq:phi-poincare-2d} than when it is simply a polynomial of the form $P(x)y^2$.

\begin{figure}[t]
\centering
\begin{tabular}{crr}
& \multicolumn{2}{c}{$\HBsos$} \\
\cline{2-3}
$\nu$ & Lens~ & Crescent \\
\hline
1 & 0.0000 & 0.0000 \\
3 & 0.0000 & 26.0860 \\
5 & 0.0000 & 35.2271 \\
7 & 14.3440 & 35.5110 \\
9 & 16.0778 & 35.5241 \\
11 & 16.1005 & 35.5256 \\
13 & 16.1009 & 35.5258 \\
\end{tabular}
\hspace{20pt}
\begin{tabular}{@{}c@{}}
\includegraphics[width=0.33\textwidth]{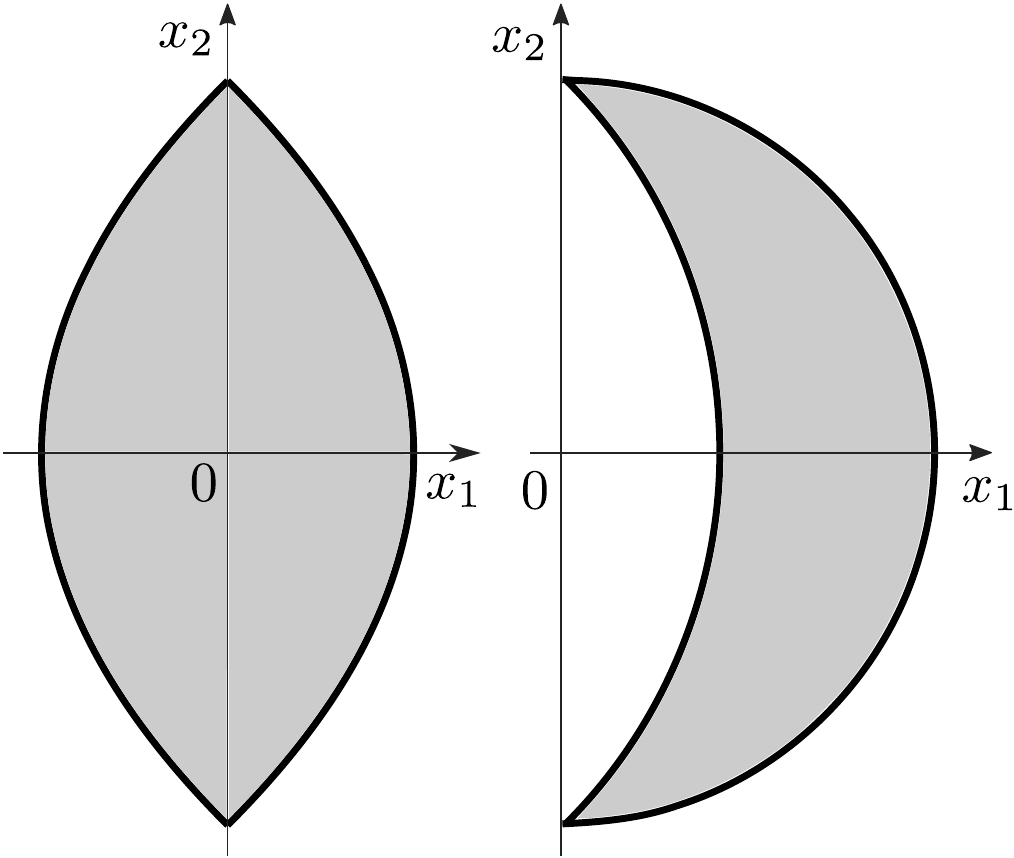}
\end{tabular}\qquad
\caption{Lower bounds $\HBsos$ on the optimal Poincar\'e constant $\lambda^*$ for the $L^2$ norm in 2D, computed by solving SOS programs to construct rational $\vphi$ with numerators of polynomial degree $\nu$ in $(x_1,x_2)$ (see text). Computations were carried out for a lens-shaped domain (left) and crescent-shaped domain (right) defined by \cref{parabolic-domain,cs-domain}, respectively. The literature gives an analytical value of $\lambda^*= (2\gamma_0)^2\approx16.100953$ for the lens domain, where $\gamma_0$ is the smallest zero of the Bessel function of the first kind $J_{-1/4}$~\cite{MorseFeshbach}. To the best of our knowledge, no analytical value is available for the crescent domain. Tabulated values have been rounded to the precision shown.}
\label{table:par-cs}
\end{figure}

\subsection{Optimal Poincar\'{e} inequality for the $L^4$ norm in 1D}
\label{sec: L4-example}

We consider problem~\cref{L2n-problem} with $2q=4$ and $\Omega=(-1,1)$, for
which $\lambda^*\approx4.566$ is known analytically~\cite{Lindquist}.
\Cref{thm:L2n} guarantees that $\HBsos\nearrow\HBpdr=\lambda^*$. To test this convergence in practice we formulate SOS relaxations using rational $\vphi$ of the form $\varphi(x,y) = P(x)y^2/(1-x^2)^3$, where the degree-$\nu$ polynomial $P$ is tunable in the SOS program (cf.\ \cref{rem:rational}).
This rational ansatz is motivated by the optimal $\vphi$ constructed in the proof of~\Cref{thm:L2n}, which suggests using $\vphi(x,y)=F(x)y^2$ with $F(x)$ having a third-order pole at the boundary points. \Cref{table:L4} reports the resulting values of $\HBsos$, which increase towards $\lambda^*$ as $\nu$ is raised. As in the previous example, rational $\vphi$ leads to faster convergence compared to purely polynomial $\vphi$.

\begin{table}[t]
\centering
\caption{Lower bounds $\HBsos$ on the optimal Poincar\'e constant $\lambda^*$ for the $L^4$ norm in 1D, computed by solving SOS programs to construct rational $\vphi$ with numerators of polynomial degree $\nu$ in $x$. The literature gives an analytical expression~\cite{Lindquist} whose value is $\lambda^*\approx4.566$. Values have been rounded to the precision shown.}
\small
\bgroup
\def\arraystretch{0.9}
\begin{tabularx}{\textwidth}{>{\centering\arraybackslash}X
		>{\centering\arraybackslash}X
		>{\centering\arraybackslash}X
		>{\centering\arraybackslash}X
		>{\centering\arraybackslash}X
		>{\centering\arraybackslash}X
		>{\centering\arraybackslash}X
		>{\centering\arraybackslash}X
		>{\centering\arraybackslash}X
		>{\centering\arraybackslash}X}
\toprule
$\nu$ & 1 & 11 & 21 & 31 & 41 & 51 & 101 & 151 & 201 \\
\midrule
$\HBsos$ & 0.000 & 4.217 & 4.479 & 4.528 & 4.540 & 4.542 & 4.548 & 4.556 & 4.566\\
\bottomrule
\end{tabularx}
\egroup
\normalsize
\label{table:L4}
\end{table}

\subsection{\label{sec:nonconvex}An example that is non-convex in the derivative}

\begin{table}[t]
\centering
\caption{Lower bounds $\Lsos$ on the infimum $\cF^*\approx0.02075$ of~\cref{Pedregal1}, computed by solving SOS programs to construct polynomial $\vphi$ with degree $\nu$ in $x$. Values have been rounded to the precision shown.}
\small
\bgroup
\def\arraystretch{0.9}
\setlength{\tabcolsep}{2pt}
\begin{tabularx}{\textwidth}{>{\centering\arraybackslash}X
		>{\centering\arraybackslash}X
		>{\centering\arraybackslash}X
		>{\centering\arraybackslash}X
		>{\centering\arraybackslash}X
		>{\centering\arraybackslash}X
		>{\centering\arraybackslash}X
		>{\centering\arraybackslash}X
		>{\centering\arraybackslash}X
		>{\centering\arraybackslash}X}
\toprule
$\nu$ & 1 & 2 & 3 & 4 & 5 & 6 & $\cdots$ & 27 & 28 \\
\midrule
$\Lsos$ & 0.00000 & 0.00780 & 0.01897 & 0.01969 & 0.02073 & 0.02074 &
$\cdots$ & 0.02074 & 0.02075  \\
\bottomrule
\end{tabularx}
\egroup
\normalsize
\label{table-ncvx}
\end{table}

Consider the following integral variational problem over $u:(0,1)\to\R$ with fixed boundary values:
\begin{equation}
\mathcal{F}^* = \inf_{\substack{u\in W^{1,2}\\ \text{s.t.}~u(0)=0\\\hspace{17pt}u(1)=\frac12}} \int_0^1 \left[(u_x^2 - 1)^2 + \tfrac12 u^2\right] \dx.
\label{Pedregal1}
\end{equation}
The integrand is non-convex in $u_x$, and the infimum is not attained: minimizing sequences of $u$ are increasingly oscillatory and do not have a limit in $W^{1,2}$~\cite{Pedregal2000}. In this 1D example the relaxed minimization over gradient Young measures has the same minimum~\cite{PedregalBook}, and a numerical method to discretize the relaxed problem appears in~\cite{Egozcue2003}. 
We have implemented the method of \cite{Egozcue2003} with increasing resolution to
find $\mathcal{F}^* \approx 0.02075$. This value is in fact equal to $\cF^*$, up to the reported precision, as verified by our lower bounds. To compute lower bounds on $\cF^*$ we use the relaxation $\Lsos$ defined in~\cref{eq:int-ineq-SOS} with $\varphi(x,y)=P(x) + Q(x)y + R(x)y^2$, where $P,Q,R$ are degree-$\nu$ polynomials that are tunable in the SOS programs. Raising $\nu$ gives lower bounds $\Lsos$ that converge quickly to 0.02075, as reported in~\cref{table-ncvx}. We note that the equality $\Lpdr=\cF^*$ is implied for this example  by \cite[Theorem 5.1 and Remark 5.2]{Fantuzzi2022}, but the numerically apparent convergence $\Lsos\nearrow\cF^*$ remains to be proved.

\section{\label{sec: con}Conclusions}

We have described a framework for finding lower bounds on global infima of integral variational problems, possibly subject to integral and pointwise constraints, by relaxing the Lagrangian dual problem into more tractable problems whose suprema give lower bounds. As proved in~\cite{Fantuzzi2022}, the relaxed maximizations obtained with this approach are weakly Lagrangian dual to the measure-theoretic relaxations described in~\cite{Korda2018}, and they are strongly dual if certain coercivity conditions hold. We have also extended the framework to more general optimization problems where variational integrals appear in constraints.

In the case of problems posed with polynomial data, further relaxation steps lead to computationally tractable sum-of-squares programs. We have proved that the solutions of such SOS programs with increasing polynomial degree converge to the exact answers for three particular classes of problems: variational problems with quadratic integrands, optimizations constrained by variational integrals that define eigenvalues of elliptic operators, and similar optimizations that define optimal constants of the $L^p$ Poincar\'e inequality in one dimension with even $p$. For several examples that fall in these classes and one that does not, we tested the SOS computational approach by computing bounds on optima using increasing polynomial degrees. In all examples, the computed bounds appear to converge to the exact answers to the original problems. The examples considered were simple enough that their optima could be found by existing analytical or numerical methods, so that it was possible to verify success of the present approach. However, our approach can 
be applied much more broadly, giving one-sided bounds on optima for problems that are not otherwise tractable. Determining when these bounds are sharp is the focus of ongoing work.

In the effort to characterize the conditions under which the present approach does and does not give sharp bounds on the optima of the original problems, our proofs of sharpness rely on semi-explicit constructions of optimizers for the relaxed dual problems in terms of solutions to corresponding Euler--Lagrange equations. A less explicit approach is likely needed to characterize sharpness more generally. The works~\cite{Fantuzzi2022,Korda2022} take some steps in this direction, but they also give explicit examples in which our relaxation approach is \textit{not} sharp. Further progress may be made by better characterizing how our relaxations relate to others in the variational analysis literature, such as the translation~\cite{Firoozye1991} or calibration~\cite[\S1.2]{Buttazzo1998} methods.

The dual relaxation approach studied here can be generalized in various ways. One way is to expand the set of admissible divergence theorem identities to include all null Lagrangians, as described in \cref{rem:null}. Another is to let the variational integrals of $u$ being studied, and the null Lagrangians employed, depend on higher derivatives than $\nabla u$; special cases of this generalization have been written down previously~\cite{Ahmadi2016,Ahmadi2017}. A third extension concerns boundary conditions on $u$ that are not of Dirichlet type. This is needed to impose periodicity of $u$, as well as for Neumann conditions that require pointwise boundary constraints to depend on $\nabla u$ rather than only on $u$. We have practically implemented such boundary conditions in several examples that will be reported in a subsequent publication, and a general approach (without sharpness results) is discussed in~\cite{Korda2018} under the assumption that $u$ is Lipschitz continuous. Extending the general framework, and proving its sharpness, for such boundary conditions with less regular $u$ is the subject of ongoing work. Lastly, it remains an open problem to find ways to improve the computational efficiency and the generality of the SOS approach we have described. Symmetries of a variational problem, for instance, can be exploited to reduce the size of SOS programs, and the same may be true of other types of problem structure. It should also be possible to adapt the SOS framework to tackle problems that are not purely polynomial but include, for example, trigonometric or exponential terms~\cite{Megretski2003,Roh2007}. A promising application area that calls for all of these extensions is the analysis of nonlinear PDEs through variational methods, where SOS methods are just starting to be employed~\cite{Chernyshenko2014a,Ahmadi2016,Goluskin2019a,Magron2018,Korda2018,Ahmadi2019,Fuentes2022,Marx2020}.

\section*{Acknowledgements}

This work has benefited from numerous insights by Ian Tobasco, including the connection to null Lagrangians described in Remark 2.1. We also acknowledge helpful discussions with Charles Doering, Andrew Wynn, Antonis Papachristodoulou, and Giorgio Valmorbida.

%% file: appendix.tex

\appendix
\section{Proof of \texorpdfstring{\cref{th:sos-variational-problem}}{Theorem \ref{th:sos-variational-problem}}}
\label{ss:proof-sos-variational}
It suffices to prove that, for any $\varepsilon>0$, there exists $\nu$ such that 
$\Lsos \geq \Lpdr - 2\varepsilon.$
Fix any feasible tuple $(\eta^\varepsilon, \vphi^\varepsilon,h^\varepsilon,\ell_1^\varepsilon,\ldots,\ell_s^\varepsilon)$ for~\cref{eq:relaxed-dual} that satisfies 
\begin{equation}
\label{eq:epsilon-suboptimality-nonpoly-VP}
\int_{\Omega} h^\varepsilon(x) \, \dx + \sum_{i=1}^s \int_{\partial\Omega_i} \ell_i^\varepsilon(x) \,\dbound \geq \Lpdr - \varepsilon.
\end{equation}
We will use $\vphi^\varepsilon,h^\varepsilon,\ell_1^\varepsilon,\ldots,\ell_s^\varepsilon$ to construct polynomials $\vphi,h,q_1,\ldots,q_s$ such that the tuple $(\eta^\varepsilon, \vphi,h,q_1,\ldots,q_s)$ is feasible for~\cref{eq:vp-SOS} for sufficiently large $\nu$ and achieves an objective value larger than $\Lpdr - 2\varepsilon$. This will imply $\Lsos \geq \Lpdr - 2\varepsilon$.

By assumption, the domain $\Omega$ and the boundary sections $\partial\Omega_i$ are compact and $\abs{\nabla g_i}$ is strictly positive on $\partial\Omega_i$. Moreover, the $(x,y)$ projection of the compact set $\Gamma$, denoted $\Gamma_{xy}$ below, must be a compact subset of $\R^n\times\R^m$. Since polynomials are dense in the space of $t$-times continuously differentiable functions on compact sets for all $t$, we can find polynomial $\vphi$, $h$, $q_1$, $\ldots$, $q_s$ that satisfy
\begin{subequations}
	\begin{align}
	\label{eq:poly-approx-1}
	\max_{i=1,\ldots,n} \left\| \vphi_i - \vphi^\varepsilon_i \right\|_{C^1(\Gamma_{xy})} &\leq \beta \varepsilon,\\
	\label{eq:poly-approx-2}
	\left\| h - h^\varepsilon + \alpha_0 \varepsilon \right\|_{C(\Omega)} &\leq \beta \varepsilon,\\
	\label{eq:poly-approx-3}
	\left\| q_i - \ell^\varepsilon_i \abs{\nabla g_i}^{-1} + \alpha_i\varepsilon \abs{\nabla g_i}^{-2} \right\|_{C(\partial\Omega_i)} &\leq \beta \varepsilon,
	\end{align}
\end{subequations}
where $\alpha_0,\ldots,\alpha_s$ and $\beta$ are strictly positive constants to be determined later.

Omitting all function arguments to simplify the notation, we may estimate
\begin{align}
\mathcal{D}\vphi^\varepsilon - \mathcal{D}\vphi
&= \sum_{i=1}^n \frac{\partial}{\partial x_i} (\vphi_i^\varepsilon - \vphi_i) +   \sum_{i=1}^n \sum_{j=1}^m \frac{\partial}{\partial y_i} (\vphi_i^\varepsilon - \vphi_i) z_{ji}\\\nonumber
&\leq \bigg( n + \sum_{i=1}^n \sum_{j=1}^m \max_{\Gamma} \vert z_{ji} \vert \bigg) \beta \varepsilon 
\qquad \forall (x,y,z) \in \Gamma.
\end{align}
Thus,
\begin{align}
f + \mathcal{D}\vphi - \eta^\varepsilon a - h
&\geq 
f + \mathcal{D}\vphi^\varepsilon - \eta^\varepsilon a - h^\varepsilon 
+ \alpha_0\varepsilon 
- \bigg( 1 + n + \sum_{i,j} \max_{\Gamma} |z_{ji}|\bigg) \beta\varepsilon
\\\nonumber
&\geq  \alpha_0 \varepsilon - C_0 \beta\varepsilon
\qquad \forall (x,y,z) \in \Gamma,
\end{align}
where $C_0:= 1 + n + \sum_{i,j} \max_{\Gamma} |z_{ji}|$ is a positive constant that does not depend on $\alpha_0$, $\beta$ or $\varepsilon$.
The second inequality holds because $(\eta^\varepsilon, \vphi^\varepsilon,h^\varepsilon,\ell_1^\varepsilon,\ldots,\ell_s^\varepsilon)$ is feasible for~\cref{eq:relaxed-dual}. Similarly, on each $\Lambda_i$ the inequality $\vphi \cdot \nabla g_i - q_i \abs{\nabla g_i}^2 \geq 
\alpha_i \varepsilon - C_i \beta \varepsilon$ holds for some positive constant $C_i$ that depends only on the particular $\partial\Omega_i$ and $g_i$ under consideration, not on $\alpha_i$, $\beta$ or $\varepsilon$. The details are omitted for brevity.

Now choose $\alpha_0 = 2C_0 \beta$ and $\alpha_i = 2C_i \beta$ for all $i=1,\ldots,s$ to obtain
\begin{subequations}
	\begin{align}
	f + \mathcal{D}\vphi - \eta^\varepsilon a - h &\geq C_0\beta \varepsilon \quad \text{ on } \Gamma,\\
	\vphi \cdot \nabla g_i - q_i \abs{\nabla g_i}^2 &\geq C_i\beta \varepsilon \quad \text{ on } \Lambda_i, \quad i = 1,\ldots,s.
	\end{align}
\end{subequations}
The right-hand sides of these inequalities are strictly positive and, by assumption, $r_0^2 - |(x,y,z)|^2 \in \mathcal{Q}(\Gamma)$ and $r_i^2 - |(x,y)|^2 \in \mathcal{Q}(\Lambda_i)$, therefore Putinar's Positivstellensatz~\cite{Putinar} guarantees that $f + \mathcal{D}\vphi - \eta^\varepsilon a - h \in \mathcal{Q}_\nu(\Gamma)$ and $\vphi \cdot \nabla g_i - q_i \abs{\nabla g_i}^2 \in \mathcal{Q}_\nu(\Lambda_i)$ for a sufficiently large integer $\nu$. In other words, the tuple $(\eta^\varepsilon,\vphi,h,q_1,\ldots,q_s)$ is feasible for the SOS program~\cref{eq:vp-SOS}. Then, we must have
\ifx\SIOPT\undefined
\else
\if\SIOPT1
\enlargethispage{\baselineskip}
\fi
\fi
\begin{align}
\label{eq:epsilon-suboptimality-final-estimate-VP}
\Lsos
&\geq \int_{\Omega} h \, \dx + \sum_{i=1}^s \int_{\partial\Omega_i} q_i \abs{\nabla g_i} \,\dbound
\\\nonumber
&\geq \int_{\Omega} h^\varepsilon\, \dx 
- (\alpha_0+\beta)\vert\Omega\vert \varepsilon
+ \sum_{i=1}^s \int_{\partial\Omega_i} 
	\ell_i^\varepsilon 
	-\beta\varepsilon \abs{\nabla g_i}
	- \alpha_i \varepsilon\abs{\nabla g_i}^{-1}
\,\dbound
\\\nonumber
&\geq \Lpdr - \bigg[ 1+ (2C_0+1)\vert\Omega\vert \beta + \beta \sum_{i=1}^s \int_{\partial\Omega_i} 
\abs{\nabla g_i}
+2C_i \abs{\nabla g_i}^{-1}
\,\dbound
\bigg] \varepsilon.
\end{align}
The second inequality follows from elementary estimates based on~\cref{eq:poly-approx-1,eq:poly-approx-2,eq:poly-approx-3}, while the last one follows from~\cref{eq:epsilon-suboptimality-nonpoly-VP} and our choices for $\alpha_0,\ldots,\alpha_s$. Choosing a sufficiently small positive value for $\beta$ 
yields $\Lsos \geq \Lpdr - 2\varepsilon$, as desired.


\section{Proof of \texorpdfstring{\cref{th:sos-convergence-noncompact}}{Theorem \ref{th:sos-convergence-noncompact}}}
\label{ss:proof-sos-convergence-noncompact}
It suffices to prove that, for each $k\geq 1$, the functions $\varphi^k,h^k,\ell_1^k,\ldots,\ell_s^k$ can be used to construct polynomials $\varphi$, $h$, $q_1$, $\ldots$, $q_s$ (not indexed by $k$ 
to lighten the notation) 
that are feasible for~\cref{eq:vp-SOS} for some integer $\nu$ and satisfy 
%
\begin{equation}\label{eq:near-optimality-approx}
\int_\Omega h \,\dx
+\sum_{i=1}^s \int_{\partial\Omega_i} q_i \abs{\nabla g_i} \,\dbound
\geq 
\int_\Omega h^k \,\dx
+\sum_{i=1}^s \int_{\partial\Omega_i} \ell_i^k \,\dbound
- \frac1k.
\end{equation}

Before defining these polynomials we make three observations. First, since $\Lambda_i = \partial\Omega_i \times \{0\}$, the constraint $\varphi \cdot \nabla g_i - q_i\abs{\nabla g_i}^2 \in \mathcal{Q}_\nu(\Lambda_i)$ is equivalent to $\varphi(x,0) \cdot \nabla g_i(x) - q_i(x)\abs{\nabla g_i(x)}^2 \in \mathcal{Q}_\nu(\partial\Omega_i)$. Second, the positive definiteness of the matrix $H_k(x)$ implies that there exist $\delta_k > 0$ such that $H_k(x) \succeq \delta_k I$ on $\Omega$, where $I$ is the identity matrix. Third, since $\varphi^k$ depends polynomially on $y$ by assumption, we can write $\varphi^k(x,y) = \sum_{\alpha \in \mathbb{N}^m} c_\alpha(x) y^\alpha$, where the sum has only finitely many terms and each coefficient $c_\alpha:\Omega \to \R^n$ is differentiable on $\Omega$, including the boundary.

We now select polynomials $h$, $p_{\alpha,1},\ldots,p_{\alpha,n}$ for each $\alpha$ and $q_1,\ldots,q_s$ that satisfy
\begin{subequations}
	\begin{align}
	\label{eq:pi-estimate}
	\| p_{\alpha,i} - c_{\alpha,i} \|_{C^1(\Omega)} &\leq \theta k^{-1},\\
	\label{eq:hk-estimate}
	\|h - h^k\|_{C(\Omega)} &\leq \theta k^{-1},\\
	\label{eq:lk-estimate}
	\left\| q_i - \ell_i^k \abs{\nabla g_i}^{-1} + 2\theta C_i k^{-1} \abs{\nabla g_i}^{-2}\right\|_{C(\partial\Omega_i)} &\leq \theta k^{-1},
	\end{align}
\end{subequations}
where $\theta$ is a sufficiently small positive scalar to be chosen, and
\begin{equation}
C_i = \max_{\partial\Omega_i}\abs{\nabla g_i}^2 + \sum_{j=1}^n\max_{\partial\Omega_i}\abs{\frac{\partial g_i}{\partial x_j}}.
\end{equation}
The polynomial $\varphi$ approximating $\varphi^k$ is then taken to be $\varphi(x,y) = \sum_{\alpha \in \mathbb{N}^m} p_\alpha(x) y^\alpha$.

Choose $\theta$ such that
\begin{equation}\label{eq:theta-condition-1}
\theta \leq \min_{i=1,\ldots,s}\bigg( \abs{\Omega} + \int_{\partial\Omega_i} \abs{\nabla g_i} + 2C_i\abs{\nabla g_i}^{-1} \,\dbound \bigg)^{-1}.
\end{equation}
Using~\cref{eq:hk-estimate}, \cref{eq:lk-estimate}, Putinar's Positivstellensatz~\cite{Putinar}, and steps analogous to those followed in the proof of \cref{th:sos-variational-problem} we can establish~\cref{eq:near-optimality-approx} and also prove that $\varphi(x,0) \cdot \nabla g_i(x) - q_i(x)\abs{\nabla g_i(x)}^2 \in \mathcal{Q}_\nu(\partial\Omega_i)$ for a sufficiently large~$\nu$. The details are omitted for brevity. The main difficulty of our proof is to show that $f + \mathcal{D}\varphi -\eta^k a  - h \in \mathcal{Q}_\nu(\Gamma)$ since $\Gamma = \Omega \times \R^m \times \R^{m\times n}$ is not compact and we cannot directly apply Putinar's Positivstellensatz.

To overcome this obstacle, recall that since $f$ is a polynomial we can write
\begin{equation}
f(x,y,z)=\sum_{\alpha \in \mathbb{N}^m}\sum_{\beta \in \mathbb{N}^m} b_{\alpha\beta}(x) y^\alpha z^\beta
\end{equation}
for some polynomials $b_{\alpha,\beta}$, where the sum is finite. Moreover, we have
\begin{equation}
\mathcal{D}\varphi^k = \sum_{\alpha \in \mathbb{N}^n} \sum_{i=1}^n y^\alpha  \frac{\partial c_{\alpha,i}}{\partial x_i}
+ \sum_{\alpha \in \mathbb{N}^n} \sum_{i=1}^n\sum_{j=1}^m \alpha_j \, c_{\alpha,i} \, y^{\alpha - \hat{e}_j} z_{ji},
\end{equation}
so $f + \mathcal{D}\varphi^k -\eta^k a- h^k$ is a polynomial in the variables $y$ and $z$ with coefficients that are linear combinations of the $x$-dependent functions $h^k$, $b_{\alpha\beta}$, $c_{\alpha,i}$ and $\frac{\partial c_{\alpha,i}}{\partial x_i}$. These coefficients can be arranged into a symmetric matrix $F_k(x)$, not necessarily positive semidefinite, that satisfies
\begin{equation}\label{eq:Fk-definition}
f(x,y,z) + \mathcal{D}\varphi^k(x,y,z) -\eta^k a(x,y,z) - h^k(x) = m(y,z)^\mathsf{T} F_k(x) m(y,z).
\end{equation}
Similarly, we can find a symmetric polynomial matrix $F(x)$ whose entries are linear combinations of the polynomials $h$, $b_{\alpha\beta}$, $p_{\alpha,i}$ and $\frac{\partial p_{\alpha,i}}{\partial x_i}$ such that
\begin{equation}
f(x,y,z) + \mathcal{D}\varphi(x,y,z) -\eta^k a(x,y,z)  - h(x) = m(y,z)^\mathsf{T} F(x) m(y,z).
\end{equation}
Since the left-hand side of~\cref{eq:Fk-definition} equals $m(y,z)^\mathsf{T} H_k(x) m(y,z)$ by assumption, there exists a symmetric matrix $G_k(x)$ that satisfies $m(y,z)^\mathsf{T} G_k(x) m(y,z) = 0$ and $H_k(x) = F_k(x) + G_k(x)$.
This matrix $G_k(x)$ can be approximated pointwise on $\Omega$ by a symmetric polynomial matrix $G(x)$ that satisfies  $m(y,z)^\mathsf{T} G(x) m(y,z) = 0$ and
$-\frac13 \delta_k I \preceq G(x)-G_k(x) \preceq \frac13\delta_k I$ for all $x$ in $\Omega$. Fix $\theta$ in~\cref{eq:pi-estimate,eq:hk-estimate} small enough that~\cref{eq:theta-condition-1} holds and also such that
$-\frac13 \delta_k I \preceq F(x)-F_k(x) \preceq \frac13\delta_k I$  on $\Omega$.
Then, the symmetric polynomial matrix $F(x)+G(x)$ is positive definite on $\Omega$, and, since $H_k(x)=F_k(x)+G_k(x) \succeq \delta_k I$ on $\Omega$,
\begin{equation}
F(x)+G(x) = H_k(x) + [F(x)-F_k(x)] + [G(x) - G_k(x)] \succeq \tfrac13 \delta_k I.
\end{equation}
We can therefore apply a matrix version of Putinar's Positivstellensatz~\cite{Scherer2006} to construct polynomial matrices $S_0(x),\ldots,S_s(x)$ such that
\begin{equation}
F(x)+G(x) = S_0(x)^\mathsf{T} S_0(x) + \sum_{i=1}^s g_i(x) S_i(x)^\mathsf{T} S_i(x),
\end{equation}
and therefore
\small
\begin{gather}
f + \mathcal{D}\varphi -\eta^k a - h 
= m^\mathsf{T} F m
= m^\mathsf{T} (F+G) m
= \underbrace{m^\mathsf{T} S_0 S_0 m}_{=:\sigma_0} + \sum_{i=1}^s g_i \underbrace{m^\mathsf{T} S_i^\mathsf{T} S_i m}_{=:\sigma_i}.
\end{gather}
\normalsize
The polynomials $\sigma_0,\ldots,\sigma_s$ are SOS by construction, showing that $f + \mathcal{D}\varphi -\eta^k a- h \in \mathcal{Q}_\nu(\Gamma)$ for $\nu$ sufficiently large. \Cref{th:sos-convergence-noncompact} is therefore proven.

\section{Proof of \texorpdfstring{\cref{th:sos-integral-inequality}}{Theorem \ref{th:sos-integral-inequality}}}
\label{ss:proof-sos-integral-inequality}
Similar to the proof of~\cref{th:sos-variational-problem}, it suffices to show that for any $\varepsilon>0$ 
there exists $\nu$ such that $\Bsos \geq \Bpdr - 2\varepsilon$. Without loss of generality, we will assume $\varepsilon$ is small enough that $b(\lambda_0) < \Bpdr - 2\varepsilon$.

We begin by constructing a tuple $(\lambda^\delta,\eta^\delta,\vphi^\delta,h^\delta,\ell_{1}^\delta,\ldots,\ell_{s}^\delta)$ that is feasible for problem \cref{eq:Bpdr}, achieves $b(\lambda^\delta)  \geq \Bpdr - 2\varepsilon$, and satisfies
\begin{equation}\label{eq:strict-positivity-estimate}
\int_{\Omega} h^\delta(x) \, \dx + \sum_{i=1}^s \int_{\partial\Omega_i} \ell_i^\delta(x) \,\dbound \geq \delta\gamma
\end{equation}
for some positive constant $\delta$ to be chosen below. To achieve these goals, fix any $(\lambda^*,\eta^*,\vphi^*$, $h^*,\ell_{1}^*,\ldots,\ell_{s}^*)$ feasible for~\cref{eq:Bpdr} that satisfies $b(\lambda^*) \geq \Bpdr - \varepsilon$, and set $\delta = \varepsilon\vert b(\lambda_0) - \Bpdr+\varepsilon\vert^{-1}$. Note that $\delta<1$ because of our assumptions on $b(\lambda_0)$ and~$\varepsilon$. The convex combination 
\begin{multline}
(\lambda^\delta,\eta^\delta,\vphi^\delta,h^\delta,\ell_{1}^\delta,\ldots,\ell_{s}^\delta) = \delta (\lambda_0,\eta_0,\vphi_0,h_0,\ell_{1,0},\ldots,\ell_{s,0})\\
+ (1-\delta)(\lambda^*,\eta^*,\vphi^*,h^*,\ell_{1}^*,\ldots,\ell_{s}^*)
\end{multline}
is therefore feasible for the convex problem~\cref{eq:Bpdr}. The linearity of the integral and inequality~\cref{eq:strict-feasibility-condition} immediately give~\cref{eq:strict-positivity-estimate}. Finally, since $b$ is concave and $b(\lambda^*) \geq \Bpdr - \varepsilon$,
\begin{align}
b(\lambda^\delta) 
&\geq \delta b(\lambda_0) + (1-\delta)b(\lambda^*)\geq \delta b(\lambda_0) + (1-\delta)\Bpdr - (1-\delta)\varepsilon = \Bpdr -2\varepsilon.
\end{align}

To prove \cref{th:sos-integral-inequality} it suffices to use the functions $\vphi^\delta,h^\delta,\ell_{1}^\delta,\ldots,\ell_{s}^\delta$ to construct polynomials  $\vphi$, $h$, $q_1$, $\ldots$, $q_s$ such that $(\lambda^\delta,\eta^\delta,\vphi,h,q_1,\ldots,q_s)$ is feasible for~\cref{eq:int-ineq-SOS} when $\nu$ is sufficiently large. This is done as in the proof of~\cref{th:sos-variational-problem}, except here we select the constant $\beta$ small enough that
\begin{align}
\delta\gamma - \bigg[ (2C_0+1)\vert\Omega\vert  + \sum_{i=1}^s \int_{\partial\Omega_i} 
\abs{\nabla g_i}
+2C_i \abs{\nabla g_i}^{-1}
\,\dbound
\bigg] \beta \varepsilon \geq 0.
\end{align}
With this assumption, estimates similar to those in~\cref{eq:epsilon-suboptimality-final-estimate-VP} can be combined with~\cref{eq:strict-positivity-estimate} to verify the last constraint in~\cref{eq:int-ineq-SOS}. The weighted SOS conditions will hold when $\nu$ is sufficiently large for the same reasons described in the proof of~\cref{th:sos-variational-problem}.